\documentclass[11pt]{amsart}

\usepackage{amscd}
\usepackage[all, knot]{xy}
\usepackage{extpfeil}
\usepackage[top=1in, bottom=1in, left=1in, right=1in]{geometry}
\xyoption{all}
\xyoption{arc}
\usepackage{hyperref}
\usepackage{graphicx}
\usepackage{amssymb}
\usepackage{epstopdf}
\usepackage{amsmath}
\def\a{\frak{a}}
\def\b{\frak{b}}
\def\m{\frak{m}}

\def\p{\frak{p}}
\def\q{\frak{q}}

\def\C{\bC}
\def\E{\bE}

\newcommand{\bE}{\mathbb{E}}

\newcommand{\I}{\mathbb{I}}

\newcommand{\Z}{\mathbb{Z}}

\def\Gstab{\Gamma^{\operatorname{stab}}}
\def\height{\operatorname{ht}}

\def\depth{\operatorname{depth}}

\def\im{\operatorname{im}}

\def\coker{\operatorname{coker}}
\def\ker{\operatorname{ker}}
\def\id{\operatorname{id}}
\def\Gid{\operatorname{Gid}}
\def\pd{\operatorname{pd}}
\def\cd{\operatorname{cd}}
\def\id{\operatorname{id}}
\def\cpr{\operatorname{cpr}}
\def\cir{\operatorname{cir}}
\def\CIR{\operatorname{CIR}}
\def\CPR{\operatorname{CPR}}
\def\prj{\operatorname{prj}}

\def\syz{\Omega^{\operatorname{prj}}}
\def\cosyz{\Omega_{\operatorname{inj}}}
\def\stabsyz{\Omega^{\operatorname{cpr}}}
\def\stabcosyz{\Omega_{\operatorname{cir}}}

\def\dirlim{\varinjlim}

\def\del{\partial}

\def\Spec{\operatorname{Spec}}

\def\Hom{\operatorname{Hom}}

\def\Ext{\operatorname{Ext}}

\def\Kac{\operatorname{K}_{\operatorname{ac}}}

\def\K{\operatorname{K}}
\def\C{\operatorname{C}}

\def\MF{\operatorname{MF}}
\def\IF{\operatorname{IF}}
\def\LF{\operatorname{LF}}
\def\mf{\operatorname{mf}}

\newcommand{\Inj}{\operatorname{Inj}}
\newcommand{\Prj}{\operatorname{Prj}}

\newcommand{\GPrj}{\operatorname{GPrj}}
\newcommand{\GInj}{\operatorname{GInj}}
\newcommand{\sGInj}{\underline{\operatorname{GInj}}}
\newcommand{\sGPrj}{\underline{\operatorname{GPrj}}}
\newcommand{\MCM}{\operatorname{MCM}}
\newcommand{\sMCM}{\underline{\operatorname{MCM}}}
\newcommand{\modR}{\operatorname{mod}R}
\newcommand{\ModR}{\operatorname{Mod}R}

\numberwithin{equation}{section}

\theoremstyle{plain} 
\newtheorem{thm}[equation]{Theorem}

\newtheorem*{propa}{Proposition A}
\newtheorem*{propb}{Proposition B}
\newtheorem*{thmc}{Theorem C}

\newtheorem{cor}[equation]{Corollary}
\newtheorem{lem}[equation]{Lemma}
\newtheorem{prop}[equation]{Proposition}

\theoremstyle{definition}
\newtheorem{defn}[equation]{Definition}
\newtheorem{example}[equation]{Example}

\theoremstyle{remark}
\newtheorem{rem}[equation]{Remark}
\newtheorem{construction}[equation]{Construction}

\newtheorem{notation}[equation]{Notation}

\title{Stable Local Cohomology} 
\author{Peder Thompson}
\address{Department of Mathematics, University of Nebraska, Lincoln, NE 68588-0130, USA}
\email{pthompson4@math.unl.edu}

\date{\today}	


\thanks{The author was partially supported by U.S. Department of Education grant P00A120068 (GAANN) and by National Science Foundation Award DMS-0966600}

\begin{document}
\maketitle
\begin{abstract}
Let $R$ be a Gorenstein local ring, $\a$ an ideal in $R$, and $M$ an $R$-module.  The local cohomology of $M$ supported at $\a$ can be computed by applying the $\a$-torsion functor to an injective resolution of $M$.  Since $R$ is Gorenstein, $M$ has a complete injective resolution, so it is natural to ask what one gets by applying the $\a$-torsion functor to it.  Following this lead, we define stable local cohomology for modules with complete injective resolutions.  This gives a functor to the stable category of Gorenstein injective modules.  We show that in many ways this behaves like the usual local cohomology functor.  Our main result is that when there is only one non-zero local cohomology module, there is a strong connection between that module and the stable local cohomology module; in fact, the latter gives a Gorenstein injective approximation of the former.
\end{abstract}

\tableofcontents

\section*{Introduction}
Let $R$ be a Gorenstein local ring with Krull dimension $d$, $\a$ an ideal in $R$, and $M$ an $R$-module.  Local cohomology of $M$ supported at $\a$ is computed by considering the $\a$-torsion functor $\Gamma_\a$ applied to an injective resolution of $M$.  In a Gorenstein ring, every module has a complete injective resolution, so it is natural to ask what one obtains by applying $\Gamma_\a$ to the complete injective resolution as opposed to the usual injective resolution.  Applying $\Gamma_\a$ to a complete injective resolution yields an acyclic complex, so taking cohomology yields nothing of interest.  Instead, given an $R$-module $M$ with a complete injective resolution $U$, we define a single module $\Gstab_\a(M)$ as the zeroeth syzygy of $\Gamma_\a(U)$.  In a Gorenstein ring, $\Gstab_\a(-):\ModR\to \sGInj(R)$ defines a functor, where $\sGInj(R)$ is the stable category of Gorenstein injective $R$-modules.

As a motivating example, we turn to maximal Cohen Macaulay (or MCM) modules over a hypersurface; recall that MCM modules correspond to matrix factorizations \cite{Eis80}.  For a local Gorenstein ring $R$, we have an induced triangulated functor $\Gstab_\a(-):\sMCM(R)\to \sGInj(R)$, where $\sMCM(R)$ is the stable category of MCM $R$-modules (see \cite{Buc86}).  Let $S$ be a regular local ring, $f$ a non-zerodivisor, $Q=S/(f)$, and $\m$ the maximal ideal of $Q$.  Then $\Gstab_\a(-):\sMCM(Q)\to \sGInj(Q)$ induces a map $-\otimes_S \Gamma_\a(D):[\mf(S,f)]\to [\IF(S,f)]$, where $D$ is a minimal injective resolution of $S$ and $[\mf(S,f)]$ and $[\IF(S,f)]$ are the homotopy categories of finitely generated matrix factorizations and injective factorizations, respectively.  For a MCM $Q$-module $M$, there exists a corresponding matrix factorization $(\xymatrix{S^r\ar@<.5ex>[r]^A& S^r\ar@<.5ex>[l]^B})$, where $\coker(A)=M$.  Then $\Gstab_\a(M)$ can be computed by considering $(\xymatrix{S^r\ar@<.5ex>[r]^A& S^r\ar@<.5ex>[l]^B})\otimes_S\Gamma_\a(D)$.  When $\a=\m$, this is just $(\xymatrix{E^r\ar@<.5ex>[r]^A& E^r\ar@<.5ex>[l]^B})$, where $E$ is the injective hull of $S/\m$, and thus $\Gstab_\m(M)$ is isomorphic to either $\ker(A:E^r\to E^r)$ or $\ker(B:E^r\to E^r)$ (depending on the parity of $\dim S$) in the stable category $\sGInj(Q)$ (i.e., isomorphic up to direct sums of injective modules). We describe this situation more generally in Proposition \ref{slc_of_mf}.

More generally for any Gorenstein ring $R$, we obtain a nice description of stable local cohomology at the maximal ideal. Classically, $H_\m^d(M)\cong M\otimes_R E_R(R/\m)$ \cite[Exercise 9.7]{24hours}. If we let $\stabsyz_dM$ be the $d$-th shift of $M$ in $\sMCM(R)$, we can give a similar result stably (Proposition \ref{slc_at_m}):
\begin{propa}\label{thma}
Let $(R,\m)$ be a Gorenstein local ring of Krull dimension $d$ and $M\in \sMCM(R)$. Then $\Gstab_\m(M)\simeq \stabsyz_dM\otimes E(R/\m)$, where $\simeq$ represents isomorphism in $\sGInj(R)$.
\end{propa}
Perhaps the next case of interest is a height $d-1$ prime ideal $\q$ of $R$.  In Proposition \ref{d_1_prime_slc_relations}, we relate $\Gstab_\m(M)$ and $\Gstab_\q(M)$ in an exact triangle in $\sGInj(R)$:
$$\Gstab_\m(M)\to \Gstab_\q(M)\to \Gstab_\q(M_\q)\to.$$
Furthermore, we have (Proposition \ref{ses_of_slc_1})
\begin{propb}
Let $R$ be a Gorenstein ring of dimension $d$, $M$ any $R$-module, $\a$ any ideal of $R$, and $x\in R$ any element.  Set $\b=(\a,x)$.  Then there exists a short exact sequence of $R$-modules
$$0\to \Gstab_\b(M)\to \Gstab_\a(M)\to \Gstab_\a(M_x)\to 0.$$
\end{propb}

If $M$ is a MCM $R$-module, recall that $\depth(\a)$ and $\cd(\a)$ are the integers representing the first and last, respectively, degrees at which $H_\a^i(M)$ is non-vanishing. In the case where $\depth(\a)=\cd(\a)$, i.e., $H_\a^i(M)=0$ for all $i\not=\depth(\a)$, we are able to relate the stable local cohomology module and the one non-zero local cohomology module (see Theorem \ref{connect} for a more general statement).  One instance where $\depth(\a)=\cd(\a)$ is when $\a$ is generated (up to radical) by a regular sequence.
\begin{thmc}
Let $R$ be a Gorenstein local ring of Krull dimension $d$.  Suppose $M\not=0$ is a MCM $R$-module, such that $\a\subset R$ is an ideal satisfying $c=\depth(\a)=\cd(\a)$.  Then there exists a short exact sequence
$$0\to H_\a^c(M)\to \Gstab_\a(\cosyz^{c}M)\oplus E_R(H_\a^c(M))\to K\to 0,$$
where $\id_RK<\infty$.  Moreover, when $0\leq c\leq t-1$, we have $\id_RK=t-c-1$ and when $c=t$, the sequence splits and $K\cong E_R(\Gstab_\a(\cosyz^tM))$.
\end{thmc}
Here $\cosyz^cM$ represents the $c$-th cosyzygy of $M$, i.e., if $M\to I$ is an injective resolution, then $\cosyz^cM=\ker(I^c\to I^{c+1})$.  

In fact, the short exact sequence of Theorem C gives a Gorenstein injective approximation of $H_\a^c(M)$, see Corollary \ref{ginj-approx}.  In particular, we have an isomorphism $H_\a^c(M)\simeq \Gstab_\a(\cosyz^cM)$ in the stable category $\sGInj(R)$.

We now give a brief outline of the paper.  In section \ref{preliminaries}, we set notation and review some basics of injective modules and Gorenstein homological algebra. 

In section \ref{resolutions_and_constructions}, we explore alternative ways of constructing ``stable'' resolutions; we develop some of the constructions, based on much of the projective analogues found in \cite{AM02}.  One of the main goals of this section is Proposition \ref{cir_hom_and_tensor} which gives a way to build complete injective resolutions from complete projective resolutions.

We define and build up the notion of stable local cohomology in section \ref{stable}.  This theory builds (in a more concrete fashion) the functor that was touched on by Stevenson in \cite{Ste14}.  Our definition appears at Definition \ref{stableloccohdefn}.  We also derive relations between stable local cohomology modules that are analogous to ones found in classical local cohomology theory; in particular, we prove Propositions A and B from above.

In section \ref{hypersurface_section}, we explore the hypersurface case.  Here we also compute some explicit stable local modules.

Finally, in section \ref{bridge}, we show there is a tight connection between stable local cohomology and classical local cohomology, at least in the case where there is only one non-zero local cohomology module.  Our main result in this direction is Theorem \ref{connect}, which we prove in this section (in particular, this proves Theorem C from above).  In fact, the stable local cohomology module will give a Gorenstein injective approximation of $H_\a^i(M)$, see Corollary \ref{ginj-approx}.\\

\noindent
{\bf Acknowledgements:} First and foremost, many thanks are owed to my advisor Mark Walker, for countless hours of conversation and advice. Additionally, conversations with many individuals in the UNL department of mathematics have been very helpful; in particular I would like to thank Luchezar Avramov, Haydee Lindo, and Tom Marley.

\section{Preliminaries}\label{preliminaries}
We first introduce notation for the categories we will be considering.
\begin{notation}
Let $\C(\ModR)$ denote the category of complexes of $R$-modules and $\K(\ModR)$ the associated homotopy category.  Here, $\ModR$ can be replaced with $\Prj R$ or $\Inj R$, representing projective modules or injective modules, respectively.  If we only want to consider finitely generated modules, we will use lower case letters, namely $\modR$ or $\prj R$.  We often will want to consider the full subcategories of acyclic complexes, which we will denote by $\Kac(-)$. 

When $R$ is Gorenstein, denote by $\sMCM(R)$ the category with the same objects as $\MCM(R)$ (the category of maximal Cohen-Macaulay $R$-modules), but with morphisms given by the following: if $M,N\in \sMCM(R)$, then 
$$\Hom_{\sMCM(R)}(M,N)=\Hom_R(M,N)/\{f:M\to N| f\text{ factors via some $P\in \prj R$}\}.$$
We call this the stable category of maximal Cohen-Macaulay $R$-modules.  Recall that in a Gorenstein ring, maximal Cohen-Macaulay (henceforth abbreviated MCM) modules coincide with finitely generated Gorenstein projective $R$-modules \cite[Corollary 10.2.7]{EJ00}.   

Likewise, $\sGInj(R)$ denotes the stable category of Gorenstein injective $R$-modules, where objects are the same as in $\GInj(R)$, (the category of Gorenstein injective modules, whose definition we recall below) and we have factored the $\Hom$ sets by those by maps that factor through an injective module.  

We will use $\simeq$ to denote isomorphism in stable categories (context should be clear) or to denote a homotopy equivalence in $\C(\ModR)$, and $\cong$ to denote isomorphism in $\ModR$ (or in $\C(\ModR)$).
\end{notation}

\subsection{Basic tools}
We call $C$ a {\em complex (of $R$-modules)} if $C$ is a $\Z$-graded $R$-module with a differential $\del$ such that $\del^2=0$.  We can either display our complexes homologically:
$$C=\cdots \to C_{i+1}\to C_i\to C_{i-1}\to \cdots$$
or cohomologically:
$$C=\cdots \to C^{i-1}\to C^i\to C^{i+1}\to \cdots$$
We say that a complex $C$ is {\em bounded on the left (resp. right)} if $C_i=0$ for $i\gg0$ or $C^i=0$ for $i\ll0$ (resp. $C_i=0$ for $i\ll0$ or $C^i=0$ for $i\gg0$).  
For two complexes $C$ and $D$, we define their tensor product $C\otimes_R D$ as the direct sum totalization of the obvious double complex and $\Hom_R(C,D)$ as the direct product totalization of the corresponding double complex (see \cite{Wei94} 2.7.1 and 2.7.4, respectively).

For a complex $C$ of $R$-modules, we denote by $\Sigma^iC$ as the complex with $(\Sigma^iC)^n=C^{n+i}$ and differential $\del_{\Sigma^iC}^n=(-1)^i\del_C^{n+i}$.  Given a complex $C$, set $Z^i(C):=\ker(C^i\to C^{i+1})$ and $\Omega_i(C):=\coker(C_{i+1}\to C_i)$. 

The {\em truncation} of a complex $C$, denoted $C^{\geq i}$, is the complex where $(C^{\geq i})^j=\begin{cases}C^j, & j\geq i\\ 0, & j<i\end{cases}$.  Similarly, we may use $C_{\geq i}$, $C^{\leq i}$, or $C_{\leq i}$.  

If $f,g:C\to D$ are two chain maps, we use $f\sim g$ to denote the existence of a homotopy from $f$ to $g$, i.e., there exists a cohomological degree $-1$ map $h:C\to D$ such that $f-g=\del_Dh+h\del_C$.  A complex $C$ is {\em contractible} if $\id_C\sim 0_C$.  A subcomplex $A$ of $C$ is {\em irrelevant} if $A^i$ is a summand of $C^i$ for each $i\in \Z$ and $A$ is contractible. 

We denote the $R$-dual of a complex $C$ by $C^*:=\Hom_R(C,R)$.  
A {\em dualizing complex} $D$ for a ring $R$ is a complex of injective modules with bounded, finitely generated cohomology, and such that the natural homothety morphism $R\to \Hom_R(D,D)$ is a quasi-isomorphism.
If $D$ is a dualizing complex for a ring $R$, then $R$ is CM if and only if $H^i(D)=0$ for $i\not=0$ \cite[1.4]{ABS05}.  Furthermore, $R$ is Gorenstein if and only if $H^i(D)=0$ for $i\not=0$ and $H^0(D)\cong R$ \cite[1.5.7]{ABS05}.

When working in a Gorenstein ring $R$, the minimal injective resolution of $R$ is a dualizing complex for $R$, which is unique up to isomorphism.  Because we can explicitly write out a minimal injective resolution of $R$, we will often assume $D$ is a particular minimal injective resolution rather than just a dualizing complex for $R$.

For the remainder of this subsection, assume $R$ is a commutative Noetherian ring.
Recall that for an $R$-module $M$, the $\a $-torsion functor $\Gamma_\a (-)$ is defined as
$$\Gamma_\a (M)=\{x\in M:\a ^nx=0\text{ for some $n$}\},$$
which yields a left exact functor \cite[7.1 and 7.2]{24hours}.  If $I$ is an injective resolution of $M$, the {\em $i$-th local cohomology module with support in $\a$ (or in $V(\a)$)} is $H_\a ^i(M):=H^i(\Gamma_\a (I))$.

Recall that over a Noetherian ring $R$, we have a decomposition of injective $R$-modules, due to Matlis \cite{Ma58}. In fact, there exists a bijection between prime ideals $\p $ of $\Spec(R)$ and indecomposable injective modules $E(R/\p )$, where $E(R/\p)=E_R(R/\p)$ denotes the injective hull of $R/\p$ over $R$.  In this way, every injective $R$-module $J$ can be uniquely (up to isomorphism) expressed as $J\cong \bigoplus_{\p \in \Spec(R)}E(R/\p )^{\alpha_\p}$.  The irreducible injective module $E(R/\p)$ is $\p$-torsion and $\p$-local \cite{Sha69}.  

It's straightforward to see that for any prime ideal $\p$ and any other ideal $\a$, we have $\Gamma_\a(E(R/\p))=\begin{cases}
E(R/\p), & \p\supseteq \a\\
0, & \p\not\supseteq \a
\end{cases}$.  
From this, it follows that if $J$ is an injective $R$-module, then $\Gamma_\a(J)$ is also injective.  In a similar way, we have $\Hom_R(R/\m ,E(R/\p ))=\begin{cases}R/\m ,& \text{if $\p =\m $}\\0,&\text{if $\p \not=\m $}\end{cases}$ \cite[Theorem A.20]{24hours}.  As a last remark about the interplay between $\Gamma_\a$ and injectives, we note that $E(\Gamma_\a(M))\cong \Gamma_\a(E(M))$.

\subsection{Gorenstein homological algebra}\label{gorenstein_homological_algebra}
Gorenstein projective and Gorenstein injective modules (intruduced and studied in \cite{EJ95}, see definitions below) over a Gorenstein ring can be thought of as acting similar to projective and injective modules over a regular local ring.  For instance, over a Gorenstein local ring $R$, all $R$-modules have both finite Gorenstein projective dimension and finite Gorenstein injective dimension \cite[4.4.8 and 6.2.7]{Chr00}.  If we assume our ring is Cohen-Macaulay with a dualizing complex, we have an important inequality:  The Gorenstein projective (Gorenstein injective) dimension of a module is always less than or equal to the projective (injective) dimension of a module, with equality holding if the projective (injectve) dimension is finite \cite[4.4.7 and 6.2.6]{Chr00}.  Immediately we see that projective (injective) modules are Gorenstein projective (Gorenstein injective).  For relevant definitions and basics for Gorenstein projective and Gorenstein injective modules, we will use primarily as references Enochs and Jenda's book \cite{EJ00} and Christensen's book \cite{Chr00}.  

\begin{defn}\cite[Definition 10.1.1]{EJ00}
An $R$-module $M$ is said to be {\em Gorenstein injective} if and only if there is a (possibly unbounded) exact complex $U$ of injective $R$-modules such that $M=Z^0(U)$ and such that for any injective $R$-module $J$, $\Hom_R(J,U)$ is exact. 

We say $M$ is {\em Gorenstein projective} if and only if there is a (possibly unbounded) exact complex $T$ of projective $R$-modules such that $M=\Omega_0(T)$ and such that for any projective $R$-module $P$, $\Hom_R(T,P)$ is exact.
\end{defn}

\begin{defn}
Let $M$ be an $R$-module.  If $\phi:E\to M$ is a homomorphism where $E$ is an injective $R$-module, then $\phi:E\to M$ is called an {\em injective precover} if $\Hom_R(J,E)\to \Hom_R(J,M)\to 0$ is exact for every injective module $J$ \cite[Definition 1.1]{EJ95}.

We call $\phi:E\to M$ an {\em injective cover} if $\phi$ is an injective precover and whenever $f:E\to E$ is linear such that $\phi\circ f=\phi$ then $f$ is an isomorphism of $E$. 

We call a complex of the form
$$\cdots \to E_1\to E_0\to M\to 0$$
an {\em injective resolvent} of $M$ if $E_0\to M$, $E_1\to \ker(E_0\to M)$, $E_i\to \ker(E_{i-1}\to E_{i-2})$ for $i\geq 2$ are all injective precovers \cite[Definition 1.3]{EJ95}.  If these maps are all injective covers, we say the complex is a {\em minimal injective resolvent} of $M$.  In this case the complex is unique up to isomorphism \cite[page 613]{EJ95}. In general an injective resolvent is unique up to homotopy \cite[page 613]{EJ95}.
\end{defn}

In general injective (pre)covers are not necessarily surjective.  For examples of injective (pre)covers, see \cite{CEJ88}.  However, we do have that an $R$-module $M$ is Gorenstein injective if and only if its minimal injective resolvent is exact and $\Ext_R^i(J,M)=0$ for $i\geq 1$ when $J$ is any injective $R$-module \cite[Corollary 2.4]{EJ95}.

Finally, an $R$-module $M$ is called {\em reduced} if it has no non-zero injective submodules \cite[page 241]{EJ00}.

\section{Complete resolutions}\label{resolutions_and_constructions}
We first introduce complete projective and complete injective resolutions.  When $R$ is Gorenstein, we briefly recall the construction of a minimal complete projective resolution of a MCM module (the situation of \cite[Construction 3.6]{AM02} which we will utilize) and more carefully go through the construction of a minimal complete injective resolution of any module (which to our knowledge doesn't explicitly appear in the literature).  With these tools, our first goal will be to construct more computationally convenient complete injective resolutions for MCM modules. 

\subsection{Minimality and complete resolutions}
For this subsection, let $R$ be a commutative noetherian ring.  We essentially follow \cite{CJ14} for definitions regarding complete resolutions.

\begin{defn}
An acyclic complex $T$ of projective $R$-modules is called a {\em totally acyclic complex of projectives} if the complex $\Hom_R(T,Q)$ is acyclic for every projective $R$-module $Q$.  An acyclic complex $U$ of injective $R$-modules is called a {\em totally acyclic complex of injectives} if the complex $\Hom_R(J,U)$ is acyclic for every injective $R$-module $J$.  When context is clear, we often just refer to either such complex as {\em totally acyclic}.
\end{defn}

\begin{rem}\label{gor_acy_totacy}
If $R$ is Gorenstein, a complex of projective (resp., injective) $R$-modules is totally acyclic if and only if it is acyclic \cite[Corollary 5.5]{IK06}. With this in mind, an $R$-module $M$ is Gorenstein projective if and only if there exists an exact complex $T$ of projective $R$-modules such that $\Omega_0(T)=M$; $M$ is Gorenstein injective if and only if there exists an exact complex $U$ of injective $R$-modules such that $Z^0(U)=M$.
\end{rem}

\subsubsection{Minimal complexes}
\begin{defn}\cite{AM02}\label{defn_of_minimal_complex}
A complex $C$ is {\em minimal} if each homotopy equivalence $\gamma:C\to C$ is an isomorphism.
\end{defn}
An equivalent condition for minimality is given in:
\begin{prop}\cite[Proposition 1.7]{AM02}
Let $C$ be a complex of $R$-modules.  Then $C$ is minimal if and only if each morphism $\gamma:C\to C$ homotopic to $\id_C$ is an isomorphism.  Additionally, if $C$ is minimal and $A$ an irrelevant subcomplex, then $A=0$. 
\end{prop}

If $M\to I$ is an injective resolution such that $I$ is minimal, then $M\to I$ is a {\em minimal injective resolution of $M$}.  Similarly, if $P\to M$ is a projective resolution such that $P$ is minimal, then $P\to M$ is a {\em minimal projective resolution of $M$}.

\begin{rem}\label{min_inj_same_as_zero_maps}
When $C$ is a complex of finitely generated projectives over a local ring, Definition \ref{defn_of_minimal_complex} is equivalent to the familiar notion of a minimal complex of free modules \cite[Proposition 8.1]{AM02}; when $C$ is an injective resolution of some module, this notion of minimality is equivalent \cite[Example 1.8]{AM02} to the essential hull notion of minimality as in \cite[Remark 3.15]{24hours}.  More explicitly, any complex of injective modules $U$ is minimal if and only if $U^i$ is the injective hull of $\ker \del_U^i$ for all $i\in \Z$ if and only if the result of applying $\Hom_R(R/\p,-)_\p$ to the morphism $\del_U^i:U^i\to U^{i+1}$ gives the zero morphism for all $i\in \Z$ and all $\p\in \Spec(R)$.
\end{rem}

\subsubsection{Complete projective resolutions}

\begin{defn}
A {\em complete projective resolution} of an $R$-module $M$ is a diagram 
$$T\xrightarrow{\tau} P\xrightarrow{\pi} M,$$
where $\tau$ and $\pi$ are chain maps, $T$ is a totally acyclic complex of projective modules, $\pi:P\to M$ is a projective resolution, and $\tau_i:T_i\to P_i$ is an isomorphism for $i\gg0$.  Such a resolution is {\em minimal} if $T$ and $P$ are minimal complexes.  Occasionally we will refer to just the complex $T$ as a complete projective resolution for $M$.
\end{defn}

The following is a special case of \cite[Construction 3.6]{AM02}.

\begin{construction}\cite[Construction 3.6]{AM02}\label{cpr_const}
Given a MCM module $M$ over a Noetherian commutative ring $R$, we construct its complete projective resolution as follows.  Let $P\to M$ be a projective resolution with differential $\del^P$.  Let $L\to M^*$ be a projective resolution with differential $\del^L$.  Apply $(-)^*$ to $L\to M^*$ to obtain $M^{**}\to L^*$.  Say $\zeta:M\to M^{**}$ is the canonical isomorphism, $\pi:P_0\to M$ is the augmentation map, and $\iota:M^{**}\to (L_0)^*$.  Define 
$$T_i=\begin{cases}P_i, & i\leq 0\\ (L_{-i-1})^*, &i<0\end{cases} \text{ and } \del_i^T=\begin{cases}\del_i^P, & i> 0\\ \iota\circ\zeta \circ \pi, & i=0\\ (\del_{-i}^L)^*,& i<0\end{cases}$$
Then $T$ is an acyclic complex of projectives and there exists a chain map $\tau: T\to P$, where $\tau_i=\id_{P_i}$ for $i\geq 0$.

If $R$ is assumed to be Gorenstein local, then $T\to P\to M$ is easily checked to be a complete projective resolution of $M$.  If, moreover, $P\to M$ and $L\to M^*$ are chosen minimally and $M$ has no non-zero free summands, then $T\to P\to M$ is a minimal complete projective resolution.
\end{construction}

\subsubsection{Complete injective resolutions}

\begin{defn}\label{defn_cir}
A {\em complete injective resolution} of an $R$-module $M$ is a diagram 
$$M\xrightarrow{\iota}I\xrightarrow{\nu} U,$$
where $\iota$ and $\nu$ are chain maps, $U$ is a totally acyclic complex of injective modules, $\iota:M\to I$ is an injective resolution, and $\nu^i:I^i\to U^i$ is an isomorphism for $i\gg0$.  A {\em minimal complete injective resolution} of $M$ is such a resolution where $I$ and $U$ are minimal complexes.  Occasionally we will refer to just the complex $U$ as a complete injective resolution for $M$.
\end{defn}

\begin{rem}\label{existence_of_cir}
For an $R$-module $M$, a complete injective resolution of $M$ exists if and only if the Gorenstein injective dimension of $M$ is finite \cite[5.2]{CJ14}. Moreover, a local Cohen Macaulay ring $R$ admitting a dualizing complex is Gorenstein if and only if every $R$-module has finite Gorenstein injective dimension \cite[Gorenstein Theorem, GID Version 6.2.7]{Chr00}.  For a local Cohen Macaulay ring $R$ admitting a dualizing complex, every $R$-module has a complete injective resolution if and only if $R$ is Gorenstein.
\end{rem} 

\begin{lem}\label{exist_unique_maps_of_cir}
Suppose $M$ and $N$ are $R$-modules with complete injective resolutions, say $M\xrightarrow{\iota_M}I\xrightarrow{\rho_M} U$ and $N\xrightarrow{\iota_N}J\xrightarrow{\rho_N}V$, respectively. If $f:M\to N$ is a map, then there exist chain maps $\phi:I\to J$ and $\widetilde{\phi}:U\to V$ making the following diagram commute:
\[\xymatrix{
M \ar[r]^{\iota_M} \ar[d]^{f}  & I\ar[r]^{\rho_M} \ar[d]^{\phi} & U\ar[d]^{\widetilde{\phi}}\\
N\ar[r]^{\iota_N} & J\ar[r]^{\rho_N} & V.
}\]
Moreover, $\phi$ and $\widetilde{\phi}$ are unique up to homotopy equivalence.
\end{lem}
\begin{proof}
The chain map $\phi$ making the square on the left commute exists and is unique up to homotopy equivalence by \cite[Comparison Theorem 2.3.7]{Wei94}.  The existence and uniqueness (up to homotopy equivalence) of $\widetilde{\phi}$ such that the square on the right also commutes follows from the Comparison Theorem for injective resolutions \cite[Comparison Theorem 2.3.7]{Wei94} and for injective resolvents \cite[page 169]{EJ00} applied to a high enough syzygy of $U\to V$.  
\end{proof}

\begin{lem}\label{cir_does_not_depend_on_choice}
Suppose $M$ and $N$ are $R$-modules with complete injective resolutions.  Suppose $M\xrightarrow{\iota_M}I\xrightarrow{\rho_M} U$ and $M\xrightarrow{\iota_M'}I'\xrightarrow{\rho_M'} U'$ are two choices of complete injective resolutions of $M$; similarly, suppose $N\xrightarrow{\iota_N}J\xrightarrow{\rho_N} V$ and $N\xrightarrow{\iota_N'}J'\xrightarrow{\rho_N'} V'$ are two choices of complete injective resolutions of $N$.  If $f:M\to N$ is a map inducing maps as in Lemma \ref{exist_unique_maps_of_cir}, then the following square commutes up to homotopy equivalence
\[\xymatrix{
U\ar[r]^{\widetilde{\phi}}\ar[d]^{\alpha}_{\simeq} & V\ar[d]^\beta_{\simeq}\\
U'\ar[r]^{\widetilde{\phi'}} & V'
}\]
where $\alpha$ and $\beta$ are the homotopy equivalences induced by Lemma \ref{exist_unique_maps_of_cir} applied to $\id_M$ and $\id_N$, respectively.
\end{lem}
\begin{proof}
Lemma \ref{exist_unique_maps_of_cir} yields the following diagram: 

\begin{equation*}
\xymatrix{%
M\ar@{->}[rr]^{f}\ar@{->}[dr]^{\iota_{M}}\ar[dddr]_{\iota_{M}'}&&N\ar@{->}[dr]^{\iota_{N}}\ar@{->}[dddr]|!{[dl];[dr]}\hole|!{[dd];[ddrr]}\hole_{\iota_{N}'}&&\\
&I\ar@{->}[rr]^{\phi}\ar[dr]^{\rho_{M}} %
\ar@{->}[dd]^{\gamma}&&
J \ar@{->}[dd]|!{[d];[d]}\hole^(.35){\delta}\ar[dr]^{\rho_{N}}
&&\\
&&U\ar@{->}[dd]_(.65){\alpha}^(.65){\simeq}
\ar@{->}[rr]^(.4){\widetilde{\phi}}&&V\ar@{->}[dd]^{\beta}_{\simeq}\\
&I'\ar@{->}[dr]_{\rho_{M}'}\ar@{->}[rr]|!{[r];[r]}\hole^(.35){\phi'}&&
J' \ar@{->} [dr]^{\rho_{N}'}&&\\
&&U'\ar@{->}[rr]^{\widetilde{\phi'}}&&V'
}
\end{equation*}
where $\gamma:I\to I'$ and $\alpha: U \to U'$ are the unique (up to homotopy) homotopy equivalences such that $\alpha\rho_M\iota_M=\rho_{M}'\iota_{M}'$ (and $\gamma\iota_M=\iota_{M}'$ and $\alpha\rho_{M}=\rho_{M}'\gamma$); $\delta:J\to J'$ and $\beta:V\to V'$ are the unique (up to homotopy) homotopy equivalences such that $\beta\rho_N\iota_N=\rho_{N}'\iota_{N}'$ (and $\delta\iota_N=\iota_{N}'$ and $\beta\rho_N=\rho_{N}'\delta$); $\widetilde{\phi}$ is the unique (up to homotopy) map such that $\widetilde{\phi}\rho_M\iota_M=\rho_N\iota_Nf$ (and $\iota_Nf=\phi\iota_M$ and $\rho_N \phi=\widetilde{\phi}\rho_M$); and $\widetilde{\phi'}$ is the unique (up to homotopy) map such that $\widetilde{\phi'}\rho_{M}'\iota_{M}'=\rho_{N}'\iota_{N}'f$ (and $\iota_{N}'f=\phi'\iota_{M}'$ and $\rho_{N}'\phi'=\widetilde{\phi'}\rho_{M}'$). Therefore we have that $\widetilde{\phi'}\alpha$ is the unique map (up to homotopy) such that $(\widetilde{\phi'}\alpha)\rho_M\iota_M=\rho_{M}'\iota_{N}'f$ (also making the intermediate diagrams commute with $\phi'\gamma$), and $\beta \widetilde{\phi}$ is the unique map (up to homotopy) such that $(\beta \widetilde{\phi}) \rho_M \iota_M=\rho_{N}'\iota_{N}'f$ (also making the intermediate diagrams commute with $\delta \phi$).  By the uniqueness of these maps, we then have that the front square commutes up to homotopy equivalence, i.e., $\widetilde{\phi'}\alpha \simeq \beta \widetilde{\phi}$ (such that this agrees with the intermediate maps where $\phi'\gamma=\delta \phi$). 
\end{proof}

\begin{prop}\label{CIR_functor}
Let $R$ be a Gorenstein ring and for each $R$-module $M$, choose a complete injective resolution $M\to I\to U$.  Then there exists a covariant functor $\CIR(-):\ModR\to \Kac(\Inj R)$ defined on objects by $\CIR(M)=U$.  Moreover, this functor does not depend on the choice of complete injective resolution up to a canonical natural isomorphism.
\end{prop}
\begin{proof}
By Lemma \ref{exist_unique_maps_of_cir}, we have that for any map $f:M\to N$ of $R$-modules, there exists a unique (up to homotopy equivalence) map $\CIR(f):\CIR(M)\to \CIR(N)$, where clearly $\CIR(-)$ respects the identity map and compositions (by appealing to uniqueness given by Lemma \ref{exist_unique_maps_of_cir}).

Moreover, Lemma \ref{cir_does_not_depend_on_choice} shows that any two families of choices of complete injective resolutions for such a functor $\CIR(-)$ yield naturally isomorphic functors, where the canonical natural isomorphism is given by Lemma \ref{cir_does_not_depend_on_choice}. 
\end{proof}

\begin{defn}
If $R$ is a Gorenstein local ring and $M$ is an $R$-module with a minimal complete injective resolution $M\to I\to U$, we define $\cir(M):=U\in \C(\ModR)$.  
By definition of minimality, $\cir(M)$ is defined uniquely up to isomorphism; however, considered as an assignment $\ModR\to \C(\ModR)$, $\cir(-)$ is {\em not} a functor since this isomorphism is non-canonical.  As an object in $\K(\ModR)$, however, $\cir(M)\simeq \CIR(M)$.
\end{defn}

\begin{rem}\label{cir_is_an_equivalence}
Recall that $\CIR(-)$ naturally factors through $\sGInj(R)$. By \cite[Proposition 4.7]{Ste14}, there is an equivalence $\xymatrix{
K_\text{ac}(\Inj R)\ar@/^/[r]^{Z^0(-)}&\sGInj(R)\ar@/^/[l]^{\CIR(-)}
}$.
\end{rem}

\begin{rem}
For an $R$-module $M$, Enochs and Jenda defined a ``complete minimal injective resolution of $M$'' to be the concatenation of the minimal injective resolvent $J\to M$ and minimal injective resolution $M\to I$ of $M$ \cite[Definition 1.8]{EJ95}.  However, in a Gorenstein ring, this complex is acyclic if and only if $M$ is Gorenstein injective \cite[Corollary 2.3]{EJ95}.  When $R$ is Gorenstein and $M$ is reduced and Gorenstein injective, this coincides with our notion of minimal complete injective resolution; when $M$ is just Gorenstein injective (not necessarily reduced), the concatenation of the minimal injective resolvent and minimal injective resolution of $M$ contains the minimal complete injective resolution (as we have defined) as a direct summand.
\end{rem}

For any $R$-module $M$, we now construct a minimal complete injective resolution of $M$.

\begin{construction}\label{construction_cir}
Assume $R$ is Gorenstein of dimension $d$ and $M$ is any $R$-module.  Let $\iota:M\to I$ be a minimal injective resolution of $M$, with differential $\del_I$ on $I$.  Fix the minimal integer $g\geq 0$ such that $\ker \del_I^g$ is reduced Gorenstein injective; such a $g$ exists and indeed is such that $g\leq d+1$ by \cite[Theorem 10.1.13]{EJ00}.  Set $G=\ker(\del_I^g)$ and $j:G\to I^g$ the canonical inclusion.  Letting $J$ be the minimal injective resolvent for $G$, which exists by \cite[Theorem 2.1]{Eno81}, we have that the augmented complex
$$\cdots \xrightarrow{\del_2^J} J_1\xrightarrow{\del_1^J} J_0\xrightarrow{\pi} G\to 0,$$
is exact by \cite[Corollary 2.4]{EJ95}.  Define the following complex
\begin{align*}
U^i=\begin{cases}
I^i, & \text{ if } i\geq g;\\
J_{g-1-i}, & \text{ if } i<g;
\end{cases}
&
\text{ and }
\del_U^i=
\begin{cases}
\del_I^i, & \text{ if }i\geq g\\
j\circ\pi, & \text{ if }i=g-1\\
\del_{g-1-i}^{J},& \text{ if } i<g-1
\end{cases}
\end{align*}

As $J$ is an injective resolvent of $G$, we have that $\pi: U^{g-1}\to G$ is an injective precover, and so there exists a map $\nu^{g-1}:I^{g-1}\to U^{g-1}$ such that $\nu^{g-1}\circ \pi$ agrees with the canonical surjection $I^{g-1}\to G$.  The map $\nu^{g-1}$ restricts to a map $\ker \del_I^{g-1} \to \ker \del_U^{g-1}$, and then we induct, using that $U^{g-i}\to\ker(\del_U^{g-i+1})$ are injective precovers for $i>1$.  Induction gives maps $\nu^i:I^i\to U^i$ for all $i<g$, making all of the squares commute in the following diagram, where we also set $\nu^i=\id_{I^i}$ for all $i\geq g$ and unlabeled maps are the obvious ones given above:

\[\xymatrix{
\cdots \ar[r] &U^{-1}\ar[r] & U^0\ar[r] & \cdots\ar[r] & U^{g-2}\ar[r] & U^{g-1}\ar[r] & U^g\ar[r] & \cdots\\
\cdots \ar[r] &0\ar[r]\ar[u] & I^0\ar[r]\ar[u]^{\nu^0} & \cdots\ar[r] & I^{g-2}\ar[r]\ar[u]^{\nu^{g-2}} & I^{g-1}\ar[r]\ar[u]^{\nu^{g-1}} &I^g\ar[r]\ar[u]^{\nu^g}_{=}&\cdots\\
}\]

With this construction, $U$ is an acyclic complex of injective modules with a map of complexes $\nu:I\to U$ such that $\nu^i$ is an isomorphism for $i\geq g$.  As $I$ and $J$ were chosen minimally, it is easy to verify that $U$ is also a minimal complex.  To see this, note that because $G$ is reduced, the proof of \cite[Proposition 10.1.11]{EJ00} shows that $Z^i(U)\to U^i$ is an essential injection for $i<g$. As $R$ is a Gorenstein ring, we obtain for free that $U$ is totally acyclic, see Remark \ref{gor_acy_totacy}.  By assumption, $M\to I$ is an injective resolution, and by construction $\nu^i:I^i\to U^i$ is an isomorphism for $i\geq g$.  Further, since $I$ and $J$ were chosen minimally, $U$ is a minimal complex.  Hence $M\xrightarrow{\iota}I\xrightarrow{\nu}U$ is a minimal complete injective resolution, with $\nu^i:I^i\to U^i$ an isomorphism for $i\geq g$. 
\end{construction}

\begin{rem}\label{cir_unique_up_to_homotopy}
We could alter this construction by not requiring $I$ or $J$ to minimal; in this case, we would not require $G=\ker(\del_I^g)$ to be reduced (such a $g\leq d$ exists by \cite[Theorem 10.1.13]{EJ00}).  Following the rest of the construction through verbatim, this gives a (not necessarily minimal) complete injective resolution of $M$.
\end{rem}

\begin{prop}\label{cir_summand_of_any_other}
Let $M$ be an $R$-module.  If $U$ is a minimal complete injective resolution of $M$ and $V$ is any other complete injective resolution of $M$, then $U$ appears (up to isomorphism) as a direct summand of $V$ with a contractible complementary summand. 
\end{prop}
\begin{proof}
There exists homotopy inverses $\alpha:U\to V$ and $\beta:V\to U$.  The minimality of $U$ implies \cite[Proposition 1.7]{AM02} that $\alpha$ is injective, $\beta$ is surjective, $\ker \beta$ is contractible, and $V=\im\alpha\oplus \ker \beta$.
\end{proof}

\subsection{Constructing complete injective resolutions}
We now move to constructing more computationally useful complete injective resolutions of MCM modules, utilizing complete projective resolutions.

\begin{rem}\label{cpr_well_def_up_to_homotopy}
Complete projective resolutions are unique up to homotopy equivalence and a map of $R$-modules $M\to N$ induces a map (which is unique up to homotopy equivalence) between their complete projective resolutions \cite[Lemma 5.3]{AM02}.  For each MCM $R$-module $M$, choose a complete projective resolution $T\to P\to M$ and set $\CPR(M)=T$; this yields a functor $\CPR(-):\MCM(R)\to \Kac(\prj R)$.  An argument dual to Lemma \ref{cir_does_not_depend_on_choice} and Proposition \ref{CIR_functor} gives that the functor $\CPR(-)$ does not depend on the choice of complete projective resolution up to a canonical natural isomorphism.
In fact, when $R$ is Gorenstein, Buchweitz shows \cite[Theorem 4.4.1]{Buc86} that $\Omega_0(-):\Kac(\prj R) \to \sMCM(R)$ is an equivalence and it easily follows that $\CPR(-):\sMCM(R)\to \Kac(\prj R)$ gives an inverse equivalence.  
If $T\to P\to M$ is a minimal complete projective resolution, set $\cpr(M)=T\in \C(\ModR)$; then $\cpr(-)$ is a well-defined assignment of a module to a complex, since minimality of $T$ implies that it is unique up to (a non-canonical) isomorphism.  Again we caution that $\cir(-)$ is {\em not} a functor since this isomorphism is non-canonical.
\end{rem}

\begin{prop}\label{cir_hom_and_tensor}
Let $R$ be Gorenstein with $\dim(R)=d$, $D$ a minimal injective resolution for $R$, and $M$ a MCM $R$-module.  If $T\xrightarrow{\tau} P\xrightarrow{} M^*$ is a complete projective resolution of $M^*$, then $M\to \Hom_R(P,D)\xrightarrow{\Hom_R(\tau,D)}\Hom_R(T,D)$ is a complete injective resolution of $M$.  In fact, $\Hom_R(\CPR((-)^*),D)$ and $\CIR(-)$ are naturally isomorphic functors $\MCM(R) \to \Kac(\Inj R)$.  
\end{prop}
\begin{proof}
Let $M$ be any MCM $R$-module and set $\CPR(M^*)=T$.  Then there exists a projective resolution $P$ such that the diagram $T\xrightarrow{\tau} P\xrightarrow{\pi} M^*$ is a complete projective resolution of $M^*$, with $\tau_i$ an isomorphism for $i\geq g$, for some fixed integer $g$.  Apply $\Hom_R(-,D)$ to this to obtain maps of complexes    
$$\Hom_R(M^*,D)\xrightarrow{\Hom(\pi,D)} \Hom_R(P,D)\xrightarrow{\Hom(\tau,D)} \Hom_R(T,D).$$ 
As $\pi$ is a quasi-isomorphism, so is $\Hom_R(\pi,D)$ by \cite[Lemma 10.7.3]{Wei94}.  Next, applying a result of Ischebeck \cite[Exercise 3.1.24]{BH98} that says in a local ring positive $\Ext$ modules vanish for a MCM module against a finitely generated module of finite injective dimension  we obtain the map $\Hom_R(M^*,R)\to \Hom_R(M^*,D)$ induced by the quasi-isomorphism $R\to D$ is also a quasi-isomorphism.  As $M$ is MCM, $M\xrightarrow{\cong}\Hom_R(M^*,R)$, so this gives $M\to \Hom_R(M^*,D)$ is a quasi-isomorphism.
Put $\iota:M\to \Hom_R(P,D)$ as the quasi-isomorphism defined by the composition of this quasi-isomorphism and $\Hom(\pi,D)$.

As $D$ is a bounded complex of injective modules and $T\in \Kac(\prj R)$, $\Hom_R(T,D)$ is an acyclic complex of injective modules.  
Also $\Hom_R(P,D)$ is a complex of injective modules such that $\Hom_R(P,D)^i=0$ for $i<0$.  As $\iota:M\to \Hom_R(P,D)$ is a quasi-isomophism, we then have that $\iota:M\to \Hom_R(P,D)$ is an injective resolution. Recall that $\tau_i$ is an isomorphism for $i\geq g$, hence $\Hom(\tau,D)^i$ is an isomorphism for $i\geq g+d$. We then have that
$$M\xrightarrow{\iota} \Hom_R(P,D)\xrightarrow{\Hom(\tau,D)}\Hom_R(T,D)$$
is a complete injective resolution of $M$. 

So, for any MCM $R$-module $M$, both $\CIR(M)$ and $\Hom_R(\CPR(M^*),D)$ are complete injective resolutions of $M$.  Proposition \ref{CIR_functor} implies $\CIR(-)$ and $\Hom_R(\CPR((-)^*),D)$ are naturally isomorphic functors $\MCM(R)\to \Kac(\Inj R)$.
\end{proof}

\begin{lem}\label{cpr_and_R_dual}
Let $R$ be a Gorenstein local ring.  Then for a MCM $R$-module $M$ with no non-zero free summands, we have 
$$(\cpr(M^*))^*\cong \Sigma^1\cpr(M),$$
in $\C(\ModR)$.
\end{lem}
\begin{proof}
Let $P\to M$ and $L\to M^*$ be minimal projective resolutions.  Then $\cpr(M)$ is the concatenation of $P$ and $\Sigma^{-1}L^*$.  Since $P$ is also a minimal projective resolution of $M^{**}$, we have $\cpr(M^*)$ is the concatenation of $L$ and $\Sigma^{-1}P^*$.  Hence $((\cpr(M^*))^*)_{\geq -1}=P$ and $((\cpr(M^*)^*)_{\leq 0}=L^*$, therefore $(\cpr(M^*))^*=\Sigma^1\cpr(M)$.
\end{proof}

\begin{prop}\label{cir_using_cpr_of_M}
Let $R$ be local Gorenstein with $\dim(R)=d$, $M$ a MCM $R$-module with no nonzero free summands, and $D$ a minimal injective resolution for $R$.  Then we have isomorphisms in $\C(\ModR)$
$$\Hom_R(\cpr(M^*),D)\cong \cpr(M^*)^*\otimes_R D\cong \Sigma^1\cpr(M)\otimes_R D,$$
and therefore these all give isomorphic complete injective resolutions of $M$.
\end{prop}
\begin{proof}
Set $T=\cpr(M^*)$ and $S=\cpr(M)$.  By \cite[Lemma 1.1]{Ish65}, \cite[proof of Theorem 4.2]{IK06}, we can see that the map $T^*\otimes_R D\xrightarrow{\cong} \Hom_R(T,D)$ is an isomorphism, giving the first isomorphism.
The second isomorphism follows from Lemma \ref{cpr_and_R_dual}.  Proposition \ref{cir_hom_and_tensor} then shows that these all give complete injective resolutions of $M$.
\end{proof}

\begin{rem}
Although the isomorphisms in Lemma \ref{cpr_and_R_dual} and Proposition \ref{cir_using_cpr_of_M} take place in $\C(\ModR)$, these are not natural in $\C(\ModR)$.  However, after passing to $\K(\ModR)$, the isomorphisms become natural.
\end{rem}

Let $R$ be a Gorenstein ring, $M$ a MCM $R$-module.  The constructions of complete injective resolutions in Proposition \ref{cir_using_cpr_of_M} are not in general minimal, even though the complete projective resolutions are chosen minimally.  To see this, consider the following:  

\begin{example}\label{example_not_nec_min_cir}
Consider the hypersurface $R=k[[x,y]]/(x^2-y^2)$, where $k$ is any algebraically closed field of characteristic not equal to 2 (this is an $A_1$ ADE singularity, see \cite{LW12}). Let $\p=(x+y)$.  Note that this is a minimal prime ideal, since $R/\p\cong k[[x]]$ and $\height(\p)=0$.  Over this ring, we consider the MCM $R$-module defined by $M=R/\p$.  We claim that the construction of the complete injective resolution of $M$ given in Proposition \ref{cir_using_cpr_of_M} is {\em not} minimal.

Since $\dim(R)=1$, we have the minimal injective resolution of $R$ is isomorphic to $D=0\to E^0\to E^1\to 0$, where $E^i=\bigoplus_{\height(\q)=i}E(R/\q)$.  

Consider the complex
\[\xymatrix{
T=&\cdots \ar[r] & R \ar[r]^{x+y} & R \ar[r]^{x-y} &{\underbracket[0pt][0.5pt]{R}_{\mathclap{\text{ degree 0}}}} \ar[r]^{x+y} & R \ar[r] & \cdots,
}\]
where we clearly have $T\cong \Sigma^1\cpr(M)$.  We show that $T\otimes_R D$ is not a minimal complex.  As $T\otimes_R D$ is a complex of injectives, showing it is not minimal is equivalent (by Remark \ref{min_inj_same_as_zero_maps}) to showing that for some prime $\q$, and some $i\in \Z$, 
$$\Hom_{R_\q}(\kappa(\q),(T_i)_\q\otimes_{R_\q}D_\q)\to \Hom_{R_\q}(\kappa(\q),(T_{i-1})_\q\otimes_{R_\q}D_\q)$$
is not the zero map.  We consider the prime $\p=(x+y)$.  Note that $D_\p=E(R/\p)$, a complex concentrated in degree 0. So it will be enough to show that for some $i\in \Z$,
$$\Hom_{R_\p}(\kappa(\p),(T_i)_\p\otimes_{R_\p}E(R/\p))\to \Hom_{R_\p}(\kappa(\p),(T_{i-1})_\p\otimes_{R_\p}E(R/\p))$$
is not the zero map.
Localizing the map $R\xrightarrow{x-y}R$ at $\p$ gives an isomorphism $R_\p\xrightarrow{\cong}R_\p$, applying $-\otimes_{R_\p}E(R/\p)$ preserves isomorphisms, hence $R_\p\otimes_{R_\p}E(R/\p)\xrightarrow{\cong} R_\p\otimes_{R_\p} E(R/\p)$ is an isomorphism.  Furthermore, $\Hom_{R_\p}(\kappa(\p),-)$ preserves isomorphisms, hence 
 $$\Hom_{R_\p}(\kappa(\p),R_\p\otimes_{R_\p}E(R/\p))\xrightarrow{\cong} \Hom_{R_\p}(\kappa(\p),R_\p\otimes_{R_\p}E(R/\p))$$
 is an isomorphism.  Therefore $T\otimes D$ is not minimal.

\end{example}

\section{Stable local cohomology}\label{stable}
Our goal of this section is to develop a stable notion of local cohomology.
We first remark that the $\a$-torsion functor takes acyclic complexes of injectives to acyclic complexes of injectives.

\begin{lem}\label{Gamma_preserves_acy}
Let $R$ be a Noetherian commutative ring.  For an ideal $\a\subset R$, if $U\in \Kac(\Inj R)$, then $\Gamma_\a(U)\in \Kac(\Inj R)$.
\end{lem}
\begin{proof}
For $U\in \Kac(\Inj R)$, the complex $\Gamma_\a(U)$ is obtained by omitting those irreducible injective modules that correspond to primes not containing $\a$, hence $\Gamma_\a(U)$ is a complex of injective modules. We need only show that $\Gamma_\a(U)$ is also acyclic.

We induct on the number of generators of $\a$.  If $\a=0$, then $\Gamma_0(U)=U$, and there's nothing to show.  
For $i>0$, assume the result holds for any ideal $\b$ generated by $i-1$ elements, i.e., $\Gamma_\b(U)\in \Kac(\Inj R)$.  Then if $\a$ can be generated by $i$ elements, we let $\b$ be the ideal generated by $i-1$ of these generators, and set $y$ to be the remaining generator of $\a$.  We then have $0\to \Gamma_y\Gamma_\b(U)\to \Gamma_\b(U)\to (\Gamma_\b(U))_y\to 0$ is degree-wise split exact, and therefore
$$0\to \Gamma_\a(U)\to \Gamma_\b(U)\to (\Gamma_\b(U))_y\to 0$$
is exact.  Since $\Gamma_\b(U)$ and $(\Gamma_\b(U))_y$ are both acyclic (the latter since localization preserves acyclicity), we obtain that $\Gamma_\a(U)$ is acyclic as well, hence $\Gamma_\a(U)\in \Kac(\Inj R)$, as desired.
\end{proof}

This immediately recovers two results of Sazeedeh:

\begin{cor}\cite[Theorem 3.2]{Saz04}
Let $R$ be a Gorenstein ring of dimension $d$ and $\a\subset R$ an ideal.  If $G\in \GInj(R)$, then $\Gamma_\a(G)\in \GInj(R)$.
\end{cor}
\begin{proof}
Let $G$ be a Gorenstein injective $R$-module.  By definition, $G$ is the zeroth syzygy of an acyclic complex $U$ of injective modules.  Since $\Gamma_\a(-)$ is left exact, $Z^0\Gamma_\a(U)=\Gamma_\a(Z^0U)$, which coincides with $\Gamma_\a(G)$ since $\Gamma_\a(U)$ is acyclic by Lemma \ref{Gamma_preserves_acy}.  Hence again by definition, $\Gamma_\a(G)$ is Gorenstein injective.
\end{proof}

\begin{cor}\cite[Theorem 3.1]{Saz04}\label{loc_coh_of_gor_inj_is_zero_in_high_degs}
If $R$ a Gorenstein ring of dimension $d$, $G$ is a Gorenstein injective $R$-module, and $\a\subset R$ is an ideal, then $H_\a^i(G)=0$ for $i>0$.
\end{cor}
\begin{proof}
Since $G$ is Gorenstein injective, it is the zeroth syzygy of an acyclic complex $U$ of injective modules.  Then $\Gamma_\a(G)$ is the zeroth syzygy of the acyclic (by Lemma \ref{Gamma_preserves_acy}) complex $\Gamma_\a(U)$ of injective modules.  For $i>0$, $H_\a^i(G)=H^i(\Gamma_\a(U^{\geq 0}))=H^i(\Gamma_\a(U))=0$.
\end{proof}

We now come to the main definition of this document:

\begin{defn}\label{stableloccohdefn}
Let $R$ be a Noetherian commutative ring and $M$ be an $R$-module that has a minimal complete injective resolution $M\to I\to U$. For an ideal $\a$ of $R$, we define the {\em stable local cohomology module of $M$ with respect to $\a$} as 
$$\Gstab_\a(M)=Z^0(\Gamma_\a(U))\in \ModR,$$
where $Z^0(-)$ represents taking the kernel of the map between the modules in cohomological degrees $0$ and $1$.  Evidently then $\Gstab_\a(M)$ is a Gorenstein injective $R$-module (by Lemma \ref{Gamma_preserves_acy}), and this module is unique up to a non-canonical isomorphism by the minimality of $U$.  Because each homomorphism of $R$-modules induces a homomorphism of their complete injective resolutions, which is unique up to homotopy equivalence, Remark \ref{cir_is_an_equivalence} shows that each homomorphism of $R$-modules $\phi:M\to M'$ induces a homomorphism in $\sGInj(R)$
$$\Gstab_\a(\phi):\Gstab_\a(M)\to \Gstab_\a(M'),$$
that is, $\Gstab_\a(-)$ defines a functor $\ModR\to \sGInj(R)$.
\end{defn}

\begin{rem}
Since complete injective resolutions are unique up to homotopy (Lemma \ref{exist_unique_maps_of_cir}), we can equivalently define $\Gstab_\a(M)=Z^0(\Gamma_\a(\CIR(M)))\in \sGInj(R)$, which we may do without further comment.  
\end{rem}

Here are a few basic properties of stable local cohomology: 
\begin{prop}\label{basic_slc_properties}
Let $M$ be an $R$-module that has a complete injective resolution.  Then
\begin{enumerate}
\item If $\sqrt{\a}=\sqrt{\b}$, then $\Gstab_\a(M)\cong \Gstab_\b(M)$.
\item Let $\{M_\lambda\}$ be a family of $R$-modules.  Then 
$$\Gstab_\a\left(\bigoplus_\lambda M_\lambda\right)\cong \bigoplus_\lambda \Gstab_\a(M_\lambda).$$
\item If $\id_RM<\infty$, then $\Gstab_\a(M)=0$.  Conversely, if $\Gstab_0(M)=0$, then $\id_RM<\infty$. 
\end{enumerate}
\end{prop}
\begin{proof}
(1) and (2) follow immediately from \cite[Proposition 7.3]{24hours}.

For (3), if $\id_RM<\infty$ and $M\to I$ is an injective resolution, then $M\to I\to 0$ is a minimal complete injective resolution, hence $\Gstab_\a(M)=0$.  Conversely, if $\Gstab_0(M)=0$, then $M$ has a minimal complete injective resolution of the form $M\to I\to 0$, and therefore $I^i=0$ for $i\gg 0$, so $\id_RM<\infty$.
\end{proof}

When $R\to S$ is a flat ring homomorphism, we have a change of rings result for stable local cohomology.
\begin{prop}
Let $R\to S$ be a ring homomorphism such that $S$ is flat as an $R$-module, $M$ is any $S$-module having a complete $S$-injective resolution, and $\a\subseteq R$ an ideal of $R$.  Then 
$$\Gstab_\a(M)\cong \Gstab_{\a S}(M).$$
\end{prop}
\begin{proof}
Recall that injective $S$-modules are injective as $R$-modules since $S$ is a flat $R$-module.  Then a complete injective resolution $\CIR(M)$ of $M$ as an $S$-module coincides with a complete injective resolution of $M$ as an $R$-module, and the result follows by definition of stable local cohomology.
\end{proof}

Before proceeding further, we consider a simple example.

\begin{example}
Let $R=\frac{k[[x]]}{(x^2)}$, where $k$ is any field.  Then $R$ is a hypersurface with $\dim(R)=0$, and so the projective and injective modules coincide.  Set $T$ as the complex of projective (and hence injective) modules $R$ with all maps multiplication by $x$:
$$T:= \cdots \xrightarrow{x} R\xrightarrow{x} R\xrightarrow{x} R\xrightarrow{x}\cdots.$$
Then $k\to T^{\geq 0}\to T$ is a complete injective resolution of $k$.  In fact, $T$ is minimal as in this case we have $R\cong E_R(k)$.  We notice that 
$$\Gstab_{(x)}(k)=Z^0\Gamma_{(x)}(T)=\ker(\Gamma_{(x)}(R)\xrightarrow{x}\Gamma_{(x)}(R))=\ker(R\xrightarrow{x} R)=k.$$
On the other hand, $\Gstab_{(x)}(R)=0$ since $\id_RR<\infty$.  
\end{example}

A motivation for calling this stable local cohomology is that $\Gstab_\a(-)$ is the composition of the stabilization functor $Z^0\CIR(-)$ and the $\a$-torsion functor.  Notice that $Z^0\CIR(-)$ is called the Gorenstein approximation functor in \cite{Kra05}.

\begin{rem}\label{iso_gor_inj_mods}
Recall that $\simeq$ denotes an isomorphism in the stable category $\sGInj(R)$ and $\cong$ denotes an isomorphism in $\ModR$.  For Gorenstein injective modules $M$ and $N$, we comment that $M\simeq N$ if and only if there exists (possibly zero) injective $R$-modules $J_1$ and $J_2$ such that $M\oplus J_1\cong N\oplus J_2$.  (In fact, if $M$ and $N$ are reduced Gorenstein injective modules, then $M\simeq N$ if and only if $M\cong N$.)
\end{rem}

In general if $M$ is a module over a Noetherian commutative ring having a complete injective resolution, $\Gstab_\a(M)$ can be difficult to compute.  We will therefore mainly restrict ourselves to working in a Gorenstein ring $R$ so that we may use the construction of a (minimal) complete injective resolution given earlier.  Restricting further to MCM modules with no nonzero free summands will allow us to use the more accessible minimal complete projective resolution of $M$ to obtain a complete injective resolution of $M$. 

\begin{lem}\label{Gamma_over_cpr}
Let $R$ be a commutative noetherian ring, $T$ be any complex of projectives, and $D$ any complex of $R$-modules. Then
$$T\otimes_R \Gamma_\a(D)\xrightarrow{\cong}\Gamma_\a(T\otimes_R D).$$
\end{lem}
\begin{proof}
For a free $R$-module $F$ and any other $R$-module $M$, it is clear that $F\otimes_R \Gamma_\a(M)\xrightarrow{\cong} \Gamma_\a(F\otimes_R M)$ since $\Gamma_\a(-)$ commutes with arbitrary direct sums \cite[Proposition 7.3]{24hours}. Consequently, if $P$ is any projective $R$-module, we have $P\otimes_R \Gamma_\a(M)\xrightarrow{\cong} \Gamma_\a(P\otimes_R M)$.  For $i,j\in \Z$, $T_i$ is a projective module and $\Gamma_\a(D_j)$ an $R$-module, hence $T_i\otimes_R \Gamma_\a(D_j)\xrightarrow{\cong} \Gamma_\a(T_i\otimes_R D_j)$.  We have a map of bicomplexes $T\otimes_R \Gamma_\a(D)\xrightarrow{}\Gamma_\a(T\otimes_R D)$, which is an isomorphism in each bidegree; totalizing yields the desired result.
\end{proof}

\begin{prop}\label{computing_slc}
Let $R$ be a Gorenstein local ring of dimension $d$, $D$ a minimal injective resolution for $R$, $M$ a MCM $R$-module with no nonzero free summands, and $\a$ an ideal of $R$. If $T:=\cpr(M^*)$ and $S:=\cpr(M)$, then 
$$Z^0\Gamma_\a(T^*\otimes_R D)\cong Z^0\Gamma_\a(\Hom_R(T,D))\cong Z^1\Gamma_\a(S\otimes_R D)\cong Z^1(S\otimes \Gamma_\a(D)),$$
and all of these coincide with $\Gstab_\a(M)$ in $\sGInj(R)$. In particular,
\begin{enumerate}
\item $\Gstab_\a(M)\simeq Z^0\Gamma_\a(T^*\otimes_R D)$,
\item $\Gstab_\a(M)\simeq Z^0\Gamma_\a(\Hom_R(T,D))$,
\item $\Gstab_\a(M)\simeq Z^1\Gamma_\a(S\otimes_R D),$ and
\item $\Gstab_\a(M)\simeq Z^1(S\otimes_R \Gamma_\a(D))$.
\end{enumerate}

\end{prop}
\begin{proof}
The $R$-module isomorphisms follow since $\Sigma^1 S\cong T^*$ and $\Hom_R(T,D)\xrightarrow{\cong} T^*\otimes_R D$ by \cite[Lemma 1.1]{Ish65} and \cite[proof of Theorem 4.2]{IK06}, and the last isomorphism is just an application of Lemma \ref{Gamma_over_cpr}. It's therefore enough to show (2), which follows by Proposition \ref{cir_hom_and_tensor}. 
\end{proof}

\begin{notation}
Suppose $R$ is a Gorenstein ring and $M$ is an $R$-module. If $R$ is local and $T\xrightarrow{\tau}P\xrightarrow{\pi} M$ is a minimal complete projective resolution of $M$, we denote the {\em $i$-th stable syzygy} of $M$ by 
$$\stabsyz_i(M):=\coker(\tau_{i+1}:T_{i+1}\to T_i)$$
for all $i\in \Z$, and the {\em $i$-th syzygy} of $M$ by 
$$\syz_i(M):=\coker(\pi_{i+1}:P_{i+1}\to P_i)$$
for $i\geq 0$.  In this case, if $M$ is a MCM $R$-module, $\stabsyz_i(M)\simeq \syz_i(M)$ for $i\geq 0$ (isomorphic in $\sMCM(R)$).

If $R$ is not necessarily local and $M\xrightarrow{\iota} I\xrightarrow{\rho} U$ is a minimal complete injective resolution of $M$, we denote the {\em $i$-th stable cosyzygy} of $M$ by 
$$\stabcosyz^{i}(M):=\ker(\rho^i:U^i\to U^{i+1})$$
for all $i\in \Z$, and the {\em $i$-th cosygygy} of $M$ by 
$$\cosyz^{i}(M):=\ker(\iota^i:I^i\to I^{i+1})$$
for $i\geq 0$.  Here, when $M$ is a Gorenstein injective $R$-module, $\stabcosyz^i(M)\simeq \cosyz^i(M)$ for $i\geq 0$ (isomorphic in $\sGInj(R)$).

Translation functors on $\sMCM(R)$ and $\sGInj(R)$ are given by $\stabsyz_{-1}$ and $\stabcosyz^{1}$, respectively, which agree with the translation functor endowed by the equivalences 
$\xymatrix{
K_\text{ac}(\prj R)\ar@/^/[r]^{\Omega_0(-)}&\sMCM(R)\ar@/^/[l]^{\CPR(-)}
}$
and
$\xymatrix{
K_\text{ac}(\Inj R)\ar@/^/[r]^{Z^0(-)}&\sGInj(R)\ar@/^/[l]^{\CIR(-)}.
}$
In their respective stable categories, note that $\stabsyz_0(-)$ and $\stabcosyz^0(-)$ are isomorphic to the identity functors.  This agrees with the triangulation spelled out as in \cite[Theorem 4.4.1]{Buc86}, where the inverse loop functor gives the shift functor on $\sMCM(R)$, i.e., an exact triangle in $\sMCM(R)$ has the form 
$$L\to M\to N\to \stabsyz_{-1}L.$$
\end{notation}

\begin{prop}\label{shifts_in_GInj}
Let $R$ be a local Gorenstein ring.  As a functor between stable categories, $\Gstab_\a(-):\sMCM(R)\to \sGInj(R)$ is triangulated.  Furthermore, for any MCM $R$-module $M$, we have an $R$-module isomorphism
$$\Gstab_\a(\stabsyz_{-i}M)\cong \stabcosyz^i\Gstab_\a(M).$$
\end{prop}
\begin{proof}
As $\Gstab_\a(-)\simeq Z^0\Gamma_\a(\CIR(-))$, it's enough to show $Z^0(-):\Kac(\Inj R)\to \sGInj(R)$, $\Gamma_\a(-):\Kac(\Inj R)\to \Kac(\Inj R)$, and $\CIR(-):\sMCM(R)\to \Kac(\Inj R)$ are triangulated functors.  The first two functors are triangulated by \cite{Ste14,Hap88} and \cite[1.5.2]{Lip09}, respectively.  Recall that $\CIR(-)$ is naturally isomorphic to $\Hom_R(\CPR((-)^*),D)$ (by Proposition \ref{cir_hom_and_tensor}), where $D$ is a minimal injective resolution for $R$.  Note that $(-)^*$ and $\Hom_R(-,D)$ are triangulated by  \cite[1.5.2 and 1.5.3]{Lip09}, resp.), and by \cite[Theorem 4.4.1]{Buc86} we have $\CPR(-)$ is triangulated.  Composing all of these pieces shows that $\Gstab_\a(-):\sMCM(R)\to \sGInj(R)$ is a triangulated functor.

For a MCM $R$-module $M$, this then gives for any $i\in \Z$ that $\Gstab_\a(\stabsyz_{-i}M)\simeq \stabcosyz^i\Gstab_\a(M),$ and as both of these modules are reduced, by Remark \ref{iso_gor_inj_mods} we can conclude they are isomorphic as $R$-modules.
\end{proof}

\begin{rem}
Recall that an equivalent way of defining (classical) local cohomology is as a direct limit. We have a natural isomorphism \cite[Theorem 7.8]{24hours}:
$$H_\a^i(M)\cong \dirlim\Ext_R^i(R/\a^n,M).$$
It is natural to ask then why we would not define stable local cohomology in an analogous way, i.e., as $\dirlim \widehat{\Ext}_R^i(R/\a^n,M),$ or whether this is naturally isomorphic to the construction above.  Quite simply, it's not; furthermore for an $R$-module $M$ that has a complete injective resolution $U$,
$$\dirlim \widehat{\Ext}_R^i(R/\a^n,M)=0$$
for all $i\in \Z$. Using the fact that $H^i(-)$ commutes with filtered limits, see \cite[Theorem 4.33 and following comments]{24hours}, we then have
\begin{align*}
\dirlim \widehat{\Ext}_R^i(R/\a^n,M)&\cong\dirlim H^i\Hom_R(R/\a^n,U)\text{, by \cite[Theorem 5.4]{CJ14},}\\
&\cong H^i\dirlim \Hom_R(R/\a^n,U)\\
&\cong H^i\Gamma_\a(U)\\
&=0,
\end{align*}
where the last equality follows because $\Gamma_\a(U)$ is acyclic (Lemma \ref{Gamma_preserves_acy}).
\end{rem}

We now examine some of the special cases of Definition \ref{stableloccohdefn}, which may shed some light on why this seems to be the best approach for such a definition.  We will end the section with some relations among stable local cohomology modules that reflect analogous results in (classical) local cohomology. 

\subsection{Stable local cohomology at the maximal ideal}
We consider first the extremal case of $\Gstab_\m(-)$, where $\m$ is the maximal ideal of the $d$-dimensional local Gorenstein ring $(R,\m)$.  Recall that in this case, for a MCM $R$-module $M$, $H_\m^d(M)\cong M\otimes_R H_\m^d(R)\cong M\otimes E_R(R/\m)$, and all other local cohomology modules vanish.  In this case $H_\m^d(M)$ is a Gorenstein injective module \cite{Saz04}, and so is already stable in the sense we are looking for. Since $H_\m^d(M)$ comes to us in degree $d$, we would therefore expect $H_\m^d(M)$ to coincide with $\stabcosyz^{d}\Gstab_\m(M)$ (in $\sGInj(R)$).  

We first find a more explicit computation for $\Gstab_\m(M)$, for $M\in \sMCM(R)$.

\begin{prop}\label{slc_at_m}
Let $R$ be a Gorenstein local ring of dimension $d$ and $M\in \sMCM(R)$.  Then for $i\in \Z$,
$$\Gstab_\m(\stabsyz_{-i}M)\simeq \stabsyz_{d-i}M\otimes E(R/\m).$$
In particular, $\Gstab_\m(M)\simeq \stabsyz_dM\otimes E(R/\m)$.
\end{prop}
\begin{proof}
Let $M\in \sMCM(R)$ and $\cpr(M)=S$.  By part (3) of Proposition \ref{computing_slc},
\begin{align*}
\Gstab_\m(M)&\simeq Z^{1}(S\otimes_R \Gamma_\m(D))\\
&\cong Z^{1}(\cdots \to \underbrace{S_{d-1}\otimes E(R/\m)}_{\text{degree $1$}} \to \underbrace{S_{d-2} \otimes E(R/\m)}_{\text{degree $2$}}\to \cdots)\\
&\simeq \stabsyz_dM\otimes E(R/\m).
\end{align*}
Then just remark that for $i\in \Z$, $\stabsyz_d\stabsyz_{-i}(M)\cong \stabsyz_{d-i}(M)$, so
$$\Gstab_\m(\stabsyz_{-i}M)\simeq \stabsyz_d\stabsyz_{-i}M\otimes E(R/\m)\cong \stabsyz_{d-i}M\otimes E(R/\m).$$
\end{proof}

\begin{rem}
Now it's easy to see that $\stabcosyz^d\Gstab_\m(M)$ and $H_\m^d(M)$ agree in the above setting.  Let $M\in \sMCM(R)$.  Then 
\begin{align*}
\stabcosyz^d\Gstab_\m(M)&\cong \Gstab_\m(\stabcosyz^{-d}M)\text{, by Proposition \ref{shifts_in_GInj},}\\
&\simeq \stabsyz_{d-d}M\otimes E(R/\m)\text{, by Proposition \ref{slc_at_m},}\\
&\simeq M\otimes E(R/\m)\\
&\cong H_\m^d(M),
\end{align*}
so stable and classical local cohomology do indeed coincide in $\sGInj(R)$ in this situation (as well as in more generality, see ahead to Corollary \ref{slc_and_classical_coincide}).  In fact, $\Gstab_\a(M)$ and $H_\m^d(M)$ are isomorphic as $R$-modules if $H_\m^d(M)$ is reduced.
\end{rem}

\subsection{Stable local cohomology at a height $d-1$ prime ideal}
Let $(R,\m )$ be a Gorenstein local ring of dimension $d$, with $\q $ a prime ideal of height $d-1$.  Let $M\in \sMCM(R)$.  In what follows, $\otimes=\otimes_R$, unless otherwise specified.  Let $T=\cpr(M)$.  By Proposition \ref{computing_slc} and Lemma \ref{Gamma_over_cpr}, we have that $\Gstab_\q(M)\simeq Z^{1}(T\otimes \Gamma_\q (D))$, where $D$ is a minimal injective resolution for $R$.  Since $T\otimes D$ is not necessarily a minimal complete injective resolution for $M$ (see Example \ref{example_not_nec_min_cir}), we will only consider $\Gstab_\q(M)\in \sGInj(R)$.  As $R$ is Gorenstein, we have $\Gamma_\q(D)\cong (\cdots0 \to E(R/\q )\xrightarrow{\del} E(R/\m)\to 0\to \cdots)$, concentrated in degrees $d-1$ and $d$, with differential induced by that of $D$.
Set $\tau$ as the differential on $T$.
Then we have $T\otimes \Gamma_\q (D)$ is the direct sum totalization of the following (commutative) double complex:
\[\xymatrix{
& 0\ar[d] & 0\ar[d] & 0\ar[d] &\\
\cdots\ar[r]&T_{d-1}\otimes E(R/\q )\ar[r]^{\tau_{d-1}\otimes1}\ar[d]^{1\otimes \del} &  T_{d-2}\otimes E(R/\q )\ar[r]^{\tau_{d-2}\otimes 1}\ar[d]^{1\otimes \del} & T_{d-3}\otimes E(R/\q )\ar[d]^{1\otimes \del} \ar[r] & \cdots\\
\cdots\ar[r]&T_{d-1}\otimes E(R/\m)\ar[r]^{\tau_{d-1}\otimes 1}\ar[d] &  T_{d-2}\otimes E(R/\m)\ar[r]^{\tau_{d-2}\otimes1}\ar[d] & T_{d-3}\otimes E(R/\m)\ar[d] \ar[r] & \cdots\\
& 0 & 0 & 0 &
}\]
Note that $T_i$ lives in cohomological degree $-i$, $E(R/\q )$ in degree $d-1$ and $E(R/\m)$ in degree $d$.  So we get that $T\otimes \Gamma_\q(D)=$
$$ \cdots \to \underbrace{\begin{matrix}T_{d}\otimes E(R/\m )\\ \oplus \\ T_{d-1}\otimes E(R/\q)\end{matrix}}_{\text{degree $0$}}\xrightarrow{\bigl(\begin{smallmatrix}
\tau_{d}\otimes 1&1\otimes \del\\ 0&\tau_{d-1}\otimes 1
\end{smallmatrix} \bigr)}\underbrace{\begin{matrix}T_{d-1}\otimes E(R/\m )\\ \oplus \\ T_{d-2}\otimes E(R/\q)\end{matrix}}_{\text{degree $1$}}\xrightarrow{\bigl(\begin{smallmatrix}
\tau_{d-1}\otimes 1&1\otimes \del\\ 0&\tau_{d-2}\otimes 1
\end{smallmatrix} \bigr)} \underbrace{\begin{matrix}T_{d-2}\otimes E(R/\m )\\ \oplus\\  T_{d-3}\otimes E(R/\q)\end{matrix}}_{\text{degree $2$}}\to \cdots$$
Hence we have (with $\simeq$ representing isomorphism in $\sGInj(R)$)
$$\Gstab_\q (M)\simeq \ker\left(\underbrace{\begin{matrix}T_{d-1}\otimes E(R/\m )\\ \oplus \\ T_{d-2}\otimes E(R/\q)\end{matrix}}_{\text{degree $1$}}\to \underbrace{\begin{matrix}T_{d-2}\otimes E(R/\m )\\ \oplus\\  T_{d-3}\otimes E(R/\q)\end{matrix}}_{\text{degree $2$}}\right),$$
and also a commuting diagram with exact rows:
\[\xymatrix{
0\ar[r] & T_{d-1}\otimes E(R/\m)\ar@{^(_->}[r]^{\bigl(\begin{smallmatrix}1 \\ 0\end{smallmatrix}\bigr)}\ar[dd]^{\tau_{d-1}\otimes 1} & {\begin{matrix}T_{d-1}\otimes E(R/\m)\\  \oplus \\ T_{d-2}\otimes E(R/\q )\end{matrix}} \ar@{->>}[r]^{\bigl(\begin{smallmatrix}0 & 1\end{smallmatrix}\bigr)} \ar[dd]^{\bigl(\begin{smallmatrix}
\tau_{d-1}\otimes 1&1\otimes \del\\ 0&\tau_{d-2}\otimes 1
\end{smallmatrix} \bigr)} & T_{d-2}\otimes E(R/\q )\ar[dd]^{\tau_{d-2}\otimes1} \ar[r] &0\\
\\
0\ar[r] & T_{d-2}\otimes E(R/\m)\ar@{^(_->}[r]^{\bigl(\begin{smallmatrix}1 \\ 0\end{smallmatrix}\bigr)} & {\begin{matrix}T_{d-2}\otimes E(R/\m) \\ \oplus\\ T_{d-3}\otimes E(R/\q )\end{matrix}}\ar@{->>}[r]^{\bigl(\begin{smallmatrix}0 & 1\end{smallmatrix}\bigr)} & T_{d-3}\otimes E(R/\q )\ar[r] &0
}\]
The snake lemma then provides an exact sequence relating the kernels and cokernels.  For any injective module $E$, by \cite[Lemma 4.5]{Mur13}, we have the kernel of $T_{d-i}\otimes E\to T_{d-i-1}\otimes E$ is $\stabsyz_{d-i+1}M\otimes E$ and the cokernel of the same map is $\stabsyz_{d-i-1}M\otimes E$.  But then note that the connecting map in the above snake diagram is zero, hence we have an induced short exact sequence of $R$-modules:  
\begin{align}\label{snake_kernels_Rmod}
0\to \stabsyz_{d}M\otimes E(R/\m) \to \ker\begin{pmatrix}
\tau_{d-1}\otimes 1&1\otimes \del\\ 0&\tau_{d-2}\otimes 1
\end{pmatrix}\to \stabsyz_{d-1}M\otimes E(R/\q )\to 0
\end{align}
(where these $R$-modules are occurring as the kernels of the vertical maps above).  

In $\sGInj(R)$, the short exact sequence \ref{snake_kernels_Rmod} of $R$-modules induces a distinguished triangle:
\begin{align}\label{snake_kernels}
\stabsyz_{d}M\otimes E(R/\m) \to \Gstab_\q (M)\to \stabsyz_{d-1}M\otimes E(R/\q )\to \stabcosyz^1(\stabsyz_{d}M\otimes E(R/\m)).
\end{align}

\begin{lem}\label{conversion_to_maximal_ideal}
Using notation from above, we have the following isomorphism in $\sGInj(R)$:
$$\stabsyz_{d-1}M\otimes E_R(R/q)\simeq \Gstab_\q(M_\q).$$
\end{lem}
\begin{proof}
Recall that $E_R(R/\q)$ is $\q$-local, and so $E_R(R/\q)\cong E_R(R/\q)_\q\cong E_{R_\q}(R_\q/\q R_\q)$, and so we have
\begin{align*}
\stabsyz_{d-1}M\otimes_R E_R(R/\q) & \cong \stabsyz_{d-1}M\otimes_R R_\q\otimes_{R_\q} E_{R_\q}(R_\q/\q R_\q)\\
&\simeq \stabsyz_{d-1}M_\q \otimes_{R_\q} E_{R_\q}(R_\q / \q R_\q)\\
&\simeq \Gstab_{\q R_\q}(M_\q),
\end{align*}
where the last isomorphism in $\sGInj(R)$ comes from applying Proposition \ref{slc_at_m} to the $(d-1)$-dimensional Gorenstein local ring $(R_\q,\q R_\q)$.  Notationally we usually just write this as $\Gstab_\q(M_\q)$ with the ideal $\q$ here understood to be taken as an ideal of $R_\q$ and $M_\q$ considered as an $R_\q$-module.
\end{proof}

\begin{prop}\label{d_1_prime_slc_relations}
Let $(R,\m )$ be a Gorenstein local ring of dimension $d$, with $\q $ a prime of height $d-1$.  Let $M\in \sMCM(R)$.  Then there exists a distinguished triangle in $\sGInj(R)$:
$$\Gstab_\m (M)\to \Gstab_\q (M)\to \Gstab_{\q }(M_\q )\to \stabcosyz^1\Gstab_\m(M).$$
\end{prop}
\begin{proof} 
Apply Proposition \ref{slc_at_m} and Lemma \ref{conversion_to_maximal_ideal} to the distinguished triangle \ref{snake_kernels} to obtain the result. 
\end{proof}

\subsection{Short exact sequence in stable local cohomology}
We now obtain a short exact sequence in stable local cohomology relating $\Gstab_\a(-)$ and $\Gstab_{(\a,x)}(-)$ where $\a$ is any ideal and $x\in R$ any element.  

\begin{rem}\label{localization_preserves_min_inj_res}
Localization preserves injective (and hence also complete injective) resolutions \cite[Corollary 1.3]{Bas62}.
\end{rem}

\begin{prop}\label{ses_of_slc_1}
Let $R$ be a Gorenstein ring of dimension $d$, $M$ any $R$-module, $\a$ any ideal of $R$ and $x\in R$ any element.  Set $\b=(\a,x)$. Then there exists a short exact sequence of $R$-modules
$$0\to \Gstab_\b(M) \to \Gstab_\a(M)\to \Gstab_\a(M_x) \to 0.$$
\end{prop}
\begin{proof}
Choose a minimal complete injective resolution $M\to I\to U$ of $M$.  We then have an exact sequence of complexes (see remarks in \cite{HT07} before Theorem 3.2):
$$0\to \Gamma_x(U)\to U\to U_x\to 0.$$
Applying $\Gamma_\a(-)$, truncating the resulting complexes at $0$, and taking cohomology gives the desired short exact sequence (noting that $U_x$ is a minimal complete injective resolution of $M_x$ by Remark \ref{localization_preserves_min_inj_res} and $\Gamma_\a\circ \Gamma_x=\Gamma_\b$).
\end{proof}

\begin{cor}
In $\sGInj(R)$, under the same hypotheses as Proposition \ref{ses_of_slc_1}, we have the following distinguished triangle:
$$\Gstab_\b(M) \to \Gstab_\a(M)\to \Gstab_\a(M_x)\to \stabcosyz^1\Gstab_\b(M).$$
\end{cor}

\subsection{Extension of Stevenson's functor}
Let $R$ be a Gorenstein ring.  Greg Stevenson considers in \cite{Ste14}, for any ideal $\a\subset R$, 
$$\Gamma_{\a}(-):\Kac(\Inj R)\to \Kac(\Inj R),$$
which takes an acyclic complex of injectives $U$ to an acyclic complex of injectives $\Gamma_{\a}(U)$ where the degree $i$ piece consists of those indecomposable injectives corresponding to primes in $V(\a)$, i.e., primes containing $\a$ (although he uses the notation $\Gamma_{V(\a)}(-)$ for $\Gamma_\a(-)$).  Via the equivalence $\Kac(\Inj R)\to \sGInj(R)$ sending $X\mapsto Z^0(X)$, he considers $\Gamma_\a(-)$ as a functor
$$\Gamma_{\a}(-):\sGInj(R)\to \sGInj(R),$$
i.e., for a Gorenstein injective module $G$ with complete injective resolution $U$, $\Gamma_{\a}(G)=Z^0\Gamma_\a(U)$.  The functor $\Gstab_\a(-)$ is a lifting of this, such that the following diagram commutes:

\[\xymatrix{
\ModR\ar[drr]^{\Gstab_\a(-)} \ar[d]_{Z^0\CIR(-)}&& \\
\sGInj(R)\ar[rr]_{\Gamma_{\a}(-)} && \sGInj(R)
}\]
i.e., for any $R$-module $M$, 
$$\Gstab_\a(M)=Z^0\Gamma_\a(\CIR(M))\cong \Gamma_\a(Z^0\CIR(M)).$$
If $G$ is a Gorenstein injective $R$-module, then $Z^0\CIR(G)\simeq G$, hence $\Gstab_\a(G)\simeq \Gamma_\a(G)$ in $\sGInj(R)$ (and $\Gstab_\a(G)\cong \Gamma_\a(G)$ if $G$ is reduced Gorenstein injective).

\section{The hypersurface case}\label{hypersurface_section}
Let $Q$ be a regular local ring, $f\in Q$ a non-zerodivisor and $R=Q/(f)$. Referring to \cite{Wal14,DM13}, we let $[\LF(Q,f)]$ denote the homotopy category of linear factorizations, and $[\mf(Q,f)]$, $[\MF(Q,f)]$, $[\IF(Q,f)]$ denote the full subcategories of finitely generated matrix factorizations, not necessarily finitely generated matrix factorizations, and injective factorizations, respectively.  Dual to the notion of MCM modules being cokernels of finitely generated matrix factorizations \cite{Eis80}, Gorenstein injective modules appear as kernels of injective factorizations.  More precisely, Walker proves the following (as this has not appeared publicly, we include his proof below):
\begin{thm}\cite{Wal14}\label{ker_IF_to_GInj_is_equiv_Walker}
For a regular ring $Q$ and non-zerodivisor $f\in Q$, the functor 
$$\ker:[\IF(Q,f)]\to \sGInj(R)$$
(that sends an object $(\xymatrix{I_1\ar@<.5ex>[r]& I_0\ar@<.5ex>[l]})$ of $\IF(Q,f)$ to $\ker(I_0\to I_1)$) is an equivalence of triangulated categories, where $R=Q/(f)$.
\end{thm}
\begin{proof}
Since an endomorphism of an injective module determined by a non-zerodivisor is surjective, the maps $\alpha$ and $\beta$ in an injective factorization $(\xymatrix{I_1\ar@<.5ex>[r]^{\alpha}& I_0\ar@<.5ex>[l]^{\beta}})$ are surjective.  In particular, this yields a short exact sequence 
$$0\to \ker(\beta)\to I_0\xrightarrow{\beta} I_1\to 0$$
over $Q$.  Since $fx=\alpha\beta x=0$ for all $x\in \ker(\beta)$, $\ker(\beta)$ is an $R$-module.   
Then $\id_Q\ker(\beta)\leq 1$ implies, by \cite[Theorem 4.2]{BM10}, that $\Gid_R\ker(\beta)\leq 0$, hence $\ker(\beta)$ is Gorenstein injective.  We obtain a functor $\IF(Q,f)\to \sGInj(R)$. This functor sends the difference of homotopic maps of injective factorizations to a map that factors through an injective module (given by the homotopy), hence we have an induced functor $\ker:[\IF(Q,f)]\to \sGInj(R)$.  

On the other hand, this functor factors though $\Kac(\Inj R)$ in the following manner.  For an injective $Q$-module $I$, define $I^R=\Hom_Q(R,I)$, clearly seen to be an injective $R$-module. Given a map $\alpha:I_1\to I_0$ of injective $Q$-modules, let $\alpha^R$ denote the induced map of $R$-modules from $I_1^R$ to $I_0^R$.  Observe that $I^R$ is a $Q$-submodule of $I$ and $\alpha^R$ is the restriction of $\alpha$. For $\I=(\xymatrix{I_1\ar@<.5ex>[r]^{\alpha}& I_0\ar@<.5ex>[l]^{\beta}})$ in $\IF(Q,f)$,
$$\I^R:=\left(\cdots \xrightarrow{\alpha^R} I_0^R\xrightarrow{\beta^R} I_1^R \xrightarrow{\alpha^R} I_0^R \xrightarrow{\beta^R}\cdots\right)$$
is an acyclic complex (since $\alpha$ and $\beta$ are surjective).  The assignment 
$$\I \mapsto \I^R$$ 
yields a functor $\IF(Q,f)\to \Kac(\Inj R)$, and it clearly preserves homotopies and hence induces a functor on the associated homotopy categories, $(-)^R:[\IF(Q,f)]\to \Kac(\Inj R)$.  The induced functor $(-)^R$ commutes with suspensions and mapping cones and hence is triangulated.  Note that as $\ker(\beta)$ is an $R$-module, $\ker(\beta)=\ker(\beta^R)$.  
Given $\I=(\xymatrix{I_1\ar@<.5ex>[r]^{\alpha}& I_0\ar@<.5ex>[l]^{\beta}})\in [\IF(Q,f)]$,
$$\ker(\I)=\ker(\beta)=\ker(\beta^R)=Z^0(\I^R),$$
yielding a commutative diagram of functors, where $Z^0:\Kac(\Inj R)\to \sGInj(R)$ is a triangulated equivalence by \cite[Proposition 4.7]{Ste14}:
\[\xymatrix{
[\IF(Q,f)]\ar[rr]^{\ker}\ar[dr]_{(-)^R} && \sGInj(R)\\
& \Kac(\Inj R) \ar[ur]_{Z^0}^{\sim}&
}\]
The triangulated structure on $\sGInj(R)$ is by definition taken to be inherited from $\Kac(\Inj R)$, and we therefore have that $\ker:[\IF(Q,f)]\to \sGInj(R)$ is a triangulated functor.  It remains to show that $\ker$ is essentially surjective and fully faithful.

Given a Gorenstein injective $R$-module $M\not=0$, it is straightforward from \cite[Theorem 4.2]{BM10} to see that $\id_Q M\leq 1$.  There then exists a $Q$-injective resolution $0\to M\to I_0\xrightarrow{\beta}I_1\to 0$ of $M$.  Since multiplication by $f$ on $M$ is $0$, there is a unique map $\alpha:I_1\to I_0$ such that $\alpha \beta$ is multiplication by $f$ on $I_0$.  Note that $\beta \alpha \beta= f\beta$ and hence $f=\beta \alpha$ since $\beta$ is surjective.  Thus $(\xymatrix{I_1\ar@<.5ex>[r]^{\alpha}& I_0\ar@<.5ex>[l]^{\beta}})$ is an object of $\IF(Q,f)$ with $\ker(\beta)=M$, hence $\ker$ is essentially surjective.

For the remainder of the proof, set $\I=(\xymatrix{I_1\ar@<.5ex>[r]^{\alpha}& I_0\ar@<.5ex>[l]^{\beta}})$ and $\I'=(\xymatrix{I_1'\ar@<.5ex>[r]^{\alpha'}& I_0'\ar@<.5ex>[l]^{\beta'}})$. Suppose $g:\ker(\beta)\to \ker(\beta')$ is a morphism in $\sGInj(R)$.
Then we may find maps $g_j:I_j\to I_j'$ for $j=0,1$ such that $\beta' g_0=g_1 \beta$.  An easy diagram chase shows that the $g_j$'s also commute with the induced maps $\alpha,\alpha'$, and hence the $g_j$'s determine a morphism of linear factorizations from $\I$ to $\I'$ with $g_0|_{\ker(\beta)}=g$. This shows $\ker$ is a full functor.

Finally, suppose $h:\I\to \I'$ is a morphism such that $h:\ker(\beta)\to \ker(\beta')$ factors through an injective $R$-module, say $J$.  We may find a $Q$-injective resolution $0\to J\to E_0\xrightarrow{\gamma} E_1\to 0$ and construct an injective factorization $\E=(\xymatrix{E_1\ar@<.5ex>[r]^{\delta}& E_0\ar@<.5ex>[l]^{\gamma}})$.  By uniqueness up to homotopy equivalence of $Q$-injective resolutions, $h_j:I_j\to I_j'$ factors through $E_j$ for $j=0,1$ (up to homotopy equivalence), and moreover, $h:\I\to \I'$ factors through $\E$ (up to homotopy equivalence).  Next, setting $E=E_Q(J)$, we claim that
$$0\to J\to E\xrightarrow{f} E\to 0$$
is also an injective resolution of $J$.  Since $f$ is a non-zerodivisor and $E$ is an injective $Q$-module, $f:E\to E$ is onto.  The only thing left to check is that $J=K:=\ker(f:E\to E)$.  We have $J\subseteq K$, since $J$ is annihilated by $f$, and it is clear that $K$ is an $R$-module.  Given any non-zero $R$-submodule $N$ of $K$, $N$ is also a $Q$-submodule of $E$ and hence, since $J\to E$ is essential, we have $N\cap J\not=0$.  This proves $J\to K$ is an essential extension of $R$-modules and hence, since $J$ is injective, $J=K$. Set $\E'=(\xymatrix{E\ar@<.5ex>[r]^{1}&E\ar@<.5ex>[l]^{f}})$ as the corresponding injective factorization.  
But $\E'$ is contractible and $\E'\simeq \E$, so $h:\I\to \I'$ factors (in the homotopy category $[\IF(Q,f)]$) through a contractible object, hence $h$ is null-homotopic, so $\ker$ is faithful.  Therefore $\ker:[\IF(Q,f)]\to \sGInj(R)$ is an equivalence of triangulated categories.
\end{proof}

When $Q$ and $R$ are as above, and $M\in \GPrj(R)$ (or, in particular when $M$ is MCM), we will compute $\Gstab_\a(M)$ by utilizing the equivalence $[\MF(Q,f)]\to \sGPrj(R)$ given in Lemma \ref{MFQf_equiv_GPrj}.  This yields a plethora of concrete examples of stable local cohomology of MCM modules over a hypersurface.  Before proceeding, we need a few lemmas.  Compare the following with \cite[Proposition 23.6]{Swa}:

\begin{lem}\label{D_Q_to_D_R}
Let $S$ be a commutative noetherian ring, $M$ an $S$-module, and $x$ a non-zerodivisor on $S$ and on $M$.  If $M\to I$ is a minimal injective resolution of $M$, then there is a canonical induced map $M/xM\to \Sigma^1\Hom_S(S/xS,I)$ and it is a minimal injective resolution of $M/xM$.  
\end{lem}
\begin{proof}
Applying $\Hom_S(S/xS,-)$ to $I$, we obtain a short exact sequence of complexes:
$$0\to \Hom_S(S/xS,I)\to I\xrightarrow{x}I\to 0.$$
Note that $\Hom_S(S/xS,I^0)=0$ (otherwise, since $I^0\cong E_R(M)$, we would have $(0:_{E_R(M)}x)\cap M\not=0$, contradicting $x$ being a non-zerodivisor on $M$).  The long exact sequence in cohomology yields a short exact sequence:
$$0\to M\xrightarrow{x} M \xrightarrow{\operatorname{}}\Ext_S^1(S/xS,M)\to 0.$$
Therefore we have a canonical injection $M/xM\cong \Ext_S^1(S/xS,M)\to \Hom_S(S/xS, I^1),$ which implies $M/xM\to \Sigma^1\Hom_S(S/xS,I)$ is an injective resolution.  Minimality of $\Hom_S(S/xS,I)$ follows by definition and minimality of $I$: for $i\geq 1$, if $0\not=N\subseteq (0:_{I^i}x)\subseteq I^i$, then $N\cap \ker(\del_I^i)\not=0$, hence $N\cap (0:_{\ker(\del_I^i)}x)\not=0$, so the complex $\Hom_S(S/xS,I)$ is minimal as well.
\end{proof}
In particular, for a regular local ring $Q$, $f\in Q$ a non-zerodivisor, and $R=Q/(f)$, if $Q\to D_Q$ is a minimal injective resolution of $Q$, then $R\to \Sigma^1\Hom_Q(R,D_Q)$ is a minimal injective resolution of $R$.

The following lemma extends the classical result that for a hypersurface $R=Q/(f)$, $\coker:[\mf(Q,f)]\to \sMCM(R)$ is an equivalence \cite[Corollary 6.3]{Eis80}.
\begin{lem}\label{MFQf_equiv_GPrj}
Let $Q$ be a regular local ring, $f\in Q$ a non-zerodivisor, and $R=Q/(f)$. Then
$$\coker:[\MF(Q,f)]\to \sGPrj(R)$$
is an equivalence, where if $P=(\xymatrix{P_1\ar@<.5ex>[r]^{d_1}& P_0\ar@<.5ex>[l]^{d_0}})\in\MF(Q,f)$, then $\coker(P)=\coker(d_1)$.
\end{lem}
\begin{proof}
We omit the proof, as it is completely analogous to the proof that $[\mf(Q,f)]\to \sMCM(R)$ is an equivalence \cite[proof of Proposition 3.7]{Orl04}, except one needs the additional fact that a nonzero Gorenstein projective $R$-module $G$ (not necessarily finitely generated) has $\pd_QG=1$ \cite[Theorem 4.1]{BM10}. 
\end{proof}

\begin{lem}\label{GPrj_equiv_KacPrj}
Let $R$ be a Gorenstein ring.  Then
$\xymatrix{\sGPrj(R)\ar@<.5ex>[r]^{\CPR(-)}& \Kac(\Prj R)\ar@<.5ex>[l]^{\Omega_0(-)}}$
is an equivalence.
\end{lem}
\begin{proof}
We mirror the proof Buchweitz gives for showing $\Omega_0:\Kac(\prj R)\to \sMCM(R)$ is an equivalence \cite[Theorem 4.4.1]{Buc86}.  By definition, if $P\in \Kac(\Prj R)$, $\Omega_0(P)$ is a Gorenstein projective $R$-module; conversely, given a Gorenstein projective $R$-module $G$, the definition implies there exists $P\in \Kac(\Prj R)$ such that $\Omega_0(P)=G$, hence $\Omega_0$ is an essentially surjective functor.

Showing $\Omega_0$ is fully faithful follows from \cite[Lemma 5.3]{AM02}: If $S,T\in \Kac(\Prj R)$ and $f:\Omega_0S\to \Omega_0T$ is any map, then there exists a unique up to homotopy map $\widetilde{f}:S\to T$ 
such that the diagram
\[\xymatrix{
S \ar[r] \ar[d]^{\widetilde{f}} &S_{\geq 0} \ar[r] \ar[d]^{\widetilde{f}_{\geq0}} &\Omega_0S\ar[d]^f\\
T \ar[r] & T_{\geq 0} \ar[r] & \Omega_0T
}\]
commutes up to homotopy,  and further, if $f:\Omega_0S\to \Omega_0T$ is an isomorphism, then $\widetilde{f}:S\to T$ is a homotopy equivalence. Since $\widetilde{f}_{-1}|_{\Omega_0S}= f$, $\Omega_0$ is full. 
On the other hand if $\alpha,\beta:S\to T$ are two maps such that their restrictions to $\Omega_0S$ agree in $\sGInj(R)$, then \cite[Lemma 5.3]{AM02} implies $\alpha$ and $\beta$ are homotopy equivalent, hence $\Omega_0$ is faithful.

It is straightforward to see that $\CPR(-)$ gives an inverse equivalence to $\Omega_0$.
\end{proof}

\begin{prop}\label{MF_equiv_IF}
Let $Q$ be a regular local ring, $f\in Q$ a non-zerodivisor, and $Q\to D_Q$ a minimal injective resolution.  Then the functor
$$-\otimes_Q D_Q:[\MF(Q,f)]\to [\IF(Q,f)]$$
is an equivalence of triangulated categories which agrees with the equivalence $-\otimes_R D_R:\Kac(\Prj R)\to \Kac(\Inj R)$ \cite{IK06}, where $R=Q/(f)$ and $D_R$ is a minimal injective resolution of $R$.
\end{prop}
\begin{proof}
Set $R=Q/(f)$. Composing the equivalence $\ker:[\IF(Q,f)]\to \sGInj(R)$ from Theorem \ref{ker_IF_to_GInj_is_equiv_Walker} with the equivalence $\CIR:\sGInj(R)\to \Kac(\Inj R)$ \cite[Proposition 4.7]{Ste14}, we have $\Hom_Q(R,-):[\IF(Q,f)]\to \Kac(\Inj R)$, and hence $\Sigma^1\Hom_Q(R,-):[\IF(Q,f)]\to \Kac(\Inj R)$, is an equivalence. Lemmas \ref{MFQf_equiv_GPrj} and \ref{GPrj_equiv_KacPrj} show that the composition $\CPR\circ\coker(-):[\MF(Q,f)]\to \Kac(\Prj R)$ is an equivalence, and we can interpret the composition $\CPR\circ\coker(-)$ as naturally isomorphic to modding out by $f$ and ``unfolding'' the injective factorization by forgetting the 2 periodicity, we denote this simply by $-\otimes_QR$.

Setting $D_R$ to be a minimal injective resolution of $R$, Iyengar and Krause show \cite[Theorem 4.2]{IK06} that $-\otimes_R D_R:\Kac(\Prj R)\to \Kac(\Inj R)$ is an equivalence.  We therefore have the following diagram, where the horizontal functors implicitly involve a forgetting of the 2 periodicity:

\[\xymatrix{
[\MF(Q,f)]\ar[rr]_{\sim}^{-\otimes_QR}\ar[d]^{-\otimes_Q D_Q} && \Kac(\Prj R)\ar[d]^{-\otimes_R D_R}_{\sim}\\
[\IF(Q,f)]\ar[rr]_{\Sigma^1\Hom_Q(R,-)}^{\sim} && \Kac(\Inj R)
}\]
We need only show this diagram commutes.  Let $\E\in [\MF(Q,f)]$.  Then 
\begin{align*}
\E\otimes_Q R\otimes_RD_R&\simeq \E\otimes_QD_R\\
&\simeq \E\otimes_Q \Sigma^1\Hom_Q(R,D_Q)\text{, by Lemma \ref{D_Q_to_D_R},}\\
&\simeq \Sigma^1\Hom_Q(R,D_Q\otimes_Q \E)\text{, since $\E$ is flat.}
\end{align*}
This shows the diagram commutes, and therefore $-\otimes_QD_Q:[\MF(Q,f)]\to [\IF(Q,f)]$ is an equivalence.
\end{proof}

Now, for a Gorenstein projective module over a hypersurface, we can equivalently compute stable local cohomology by applying $\Gamma_\a$ to the kernel of one of the maps of the corresponding injective factorization via this equivalence.  More precisely, we have
\begin{prop}\label{slc_of_mf}
Let $Q$ be a regular local ring, $f\in Q$ a non-zerodivisor, and $R=Q/(f)$.  If $M\in \sGPrj(R)$ is an $R$-module with corresponding matrix factorization $\E\in [\MF(Q,f)]$, we have 
$$\Gstab_\a(M)\simeq \ker(\E\otimes_Q \Gamma_\a(D_Q))\in \sGInj(R),$$
where $D_Q$ is a minimal injective resolution of $Q$.
\end{prop}
\begin{proof}
For $M\in \sGPrj(R)$, Lemma \ref{MFQf_equiv_GPrj} allows us to find $\E=(\xymatrix{P_1\ar@<.5ex>[r]^A& P_0\ar@<.5ex>[l]^B})\in [\MF(Q,f)]$ with $\coker A=M$.  Then
\begin{align*}
\Gstab_\a(M)&\simeq Z^1\Gamma_\a(\CPR(M)\otimes_R D_R)\text{, by Proposition \ref{computing_slc},}\\
&\simeq Z^1\Gamma_\a(\E\otimes_Q R \otimes_R  D_R)\\
&\simeq Z^1\Gamma_\a(\Sigma^1\Hom_Q(R,\E\otimes_Q D_Q))\text{, by the proof of Proposition \ref{MF_equiv_IF},}\\
&\simeq Z^0\Gamma_\a(\Hom_Q(R,\E\otimes_Q D_Q))\text{, by 2-periodicity,}\\
&\simeq \Gamma_\a Z^0\Hom_Q(R,\E\otimes_Q D_Q)\\
&\simeq \Gamma_\a Z^0(\E\otimes_Q D_Q)\\
&\simeq \ker(\E\otimes_Q \Gamma_\a(D_Q)).
\end{align*}
\end{proof}

In particular, if $(\xymatrix{Q^r\ar@<.5ex>[r]^A& Q^r\ar@<.5ex>[l]^B})\in [\mf(Q,f)]$ and $\coker(A)=M$ (i.e., $M$ is MCM), Proposition \ref{slc_of_mf} allows us to easily compute $\Gstab_\m(M)$, where we use $\m$ to denote the maximal ideal of both $R$ and $Q$. Note that $\Gamma_\m(D_Q)\cong \Sigma^{-\dim Q}E_Q(Q/\m)$, and $(\xymatrix{Q^r\ar@<.5ex>[r]^A& Q^r\ar@<.5ex>[l]^B})\otimes_Q E_Q(Q/\m)=(\xymatrix{E_Q(Q/\m)^r\ar@<.5ex>[r]^A& E_Q(Q/\m)^r\ar@<.5ex>[l]^B})$, hence 
\begin{align*}
\Gstab_\m(M)&\simeq Z^{\dim Q}(\xymatrix{E_Q(Q/\m)^r\ar@<.5ex>[r]^A& E_Q(Q/\m)^r\ar@<.5ex>[l]^B})\\
&\simeq \begin{cases} \ker(A:E_Q(Q/\m)^r\to E_Q(Q/\m)^r), & \text{if $\dim Q$ is odd,}\\ \ker(B:E_Q(Q/\m)^r\to E_Q(Q/\m)^r), & \text{if $\dim Q$ is even}\end{cases}\\
&\simeq \begin{cases} \ker(A:E_R(R/\m)^r\to E_R(R/\m)^r), & \text{if $\dim Q$ is odd,}\\ \ker(B:E_R(R/\m)^r\to E_R(R/\m)^r), & \text{if $\dim Q$ is even.}\end{cases}
\end{align*}

\begin{example}
Consider the isolated singularity $R=\frac{k[[x,y]]}{(xy)}$, where $k$ is a field of characteristic 0.  Set $\m=(x,y)R$ and $E:=E_R(R/\m)$.  Then if we set $M=R/(x)$, we can see that $M$ is a MCM $R$-module coming from the matrix factorization $(\xymatrix{k[[x,y]]\ar@<.5ex>[r]^x& k[[x,y]]\ar@<.5ex>[l]^y})$, and have $\syz_1M\cong R/(y)$.  By Proposition \ref{slc_of_mf} and the following remarks, we have
$$\Gstab_\m(M)\simeq \ker(E\xrightarrow{y} E)\cong E/(y)E$$
and
$$\Gstab_\m(\syz_1M)\simeq \ker(E\xrightarrow{x}E)\cong E/(x)E.$$
(Alternatively, this can be seen by using Proposition \ref{slc_at_m}.)
In fact, as the complex
$$\cdots \xrightarrow{x} E\xrightarrow{y} E \xrightarrow{x} E \xrightarrow{y} \cdots$$
is minimal, $E/(y)E$ and $E/(x)E$ are reduced $R$-modules, hence we obtain isomorphisms as $R$-modules: 
$$\Gstab_\m(R/(x))\cong E/(y)E \text{ and }\Gstab_\m(R/(y))\cong E/(x)E.$$
Even more explicitly, recall that we can describe $E$ as the $k$-vector space spanned by $x^iy^j$ for $i,j\leq -1$, and with a natural $R$-module structure (for $x^my^n\in R$ and $x^iy^j\in E$, $x^my^n\cdot x^iy^j=x^{m+i}y^{n+j}$ if $m+i\leq -1$ and $n+j\leq -1$, and $=0$ otherwise, see \cite[proof of Proposition 2.3]{Lyu93}).  We write this as $k\langle x^iy^j\rangle_{i,j\leq -1}$.  In this way, we can see that 
$$\Gstab_\m(R/(x))\cong k\langle x^iy^{-1}\rangle_{i\leq -1}\text{ and }\Gstab_\m(R/(y))\cong k\langle x^{-1}y^{j}\rangle_{j\leq -1},$$
both given the $R$-module structure described above.
\end{example}

\section{A bridge between stable and classical local cohomology}\label{bridge}
Before stating and proving our main connection between stable local cohomology and classical local cohomology, we present a lemma about the structure of minimal injective resolutions of a module $M$.  For an ideal $\a\subseteq R$ we define the {\em $\a$-depth} of a (not necessarily finitely generated) module $M$ to be $\depth(\a,M):=\inf\{i|H_\a^i(M)\not=0\}$.  By \cite{FI01}, $\depth(\a,M)$ coincides with $\inf\{j|\Ext_R^j(R/\a,M)\not=0\}$.  In particular, if $(R,\m)$ is a local ring, we say the {\em depth} of $M$ is $\depth(M):=\depth(\m,M)$.  We also define the {\em cohomological dimension} of $M$ at $\a$ to be $\cd(\a,M):=\sup\{i|H_\a^i(M)\not=0\}$.  For $\p\in \Spec(R)$, for convenience we set $\kappa(\p)=R_\p/\p R_\p$. Finally, we define the {\em $i$-th Bass number of $M$} with respect to $\p\in \Spec(R)$ as $\mu_R^i(\p,M)=\dim_{\kappa(\p)}\Ext_{R_\p}^i(\kappa(\p),M_\p)$.

\begin{lem}\label{mu_vanishes_below_grade}
Let $R$ be a commutative Noetherian ring of dimension $d$ and $M$ an $R$-module.  For any $\p\in \Spec(R)$, if $i<\depth(\p,M)$, then $\mu_R^i(\p,M)=0$.
\begin{proof}
Note that if $H_\p^j(M)_\p\not=0$, then $H_\p^j(M)\not=0$.  Since $H_{\p R_\p}^j(M_\p)\cong H_\p^j(M)_\p$, we obtain $\depth(\p,M)\leq \depth(\p R_\p, M_\p)$.  Recall also, see \cite[Theorem 9.1]{24hours}, that
$$\depth(\p R_\p, M_\p)=\inf\{j|\Ext_{R_\p}^j(\kappa(\p), M_\p)\not=0\}.$$
Therefore, for $i<\depth(\p,M)\leq \depth(\p R_\p, M_\p)$, we have that $\Ext_{R_\p}^i(\kappa(\p), M_\p)=0$, hence
$$\mu_R^i(\p,M)=\dim_{\kappa(\p)} \Ext_{R_\p}^i(\kappa(\p), M_\p)=0,$$
as desired.
\end{proof}
\end{lem}

\begin{thm}\label{connect}
Let $(R,\m)$ be a Gorenstein local ring of Krull dimension $d$.  Suppose $M\not=0$ is an $R$-module where $\Gid_RM=\depth M$
and $\a\subset R$ is an ideal satisfying $c=\depth(\a,M)=\cd(\a,M)$.  Set $\Gid_R M=t$.  Then there exists a short exact sequence
$$0\to H_\a^c(M)\to \Gstab_\a(\cosyz^{c}M)\oplus E_R(H_\a^c(M))\to K\to 0,$$
where $\id_RK<\infty$.  Moreover, when $0\leq c\leq t-1$, we have $\id_RK=t-c-1$ and when $c=t$, the sequence splits and $K\cong E_R(\Gstab_\a(\cosyz^tM))$.
\end{thm}
\begin{rem}
An $R$-module $M\not=0$ that is a cosyzygy of a MCM $R$-module satisfies $\depth(M)=\Gid_R(M)$, so cosyzygies of MCM modules are candidates for modules satisfying the conditions of the theorem.  Being a cosyzygy of a MCM $R$-module means that we can write $M=\cosyz^n M'$ for some MCM $R$-module $M'$, hence $\Ext_R^i(R/\m,M)\cong \Ext_R^{i+n}(R/\m,M')$.
We have, since $\Gid(M')=\depth(R)=\depth(M')$ by \cite[Theorem 6.2.15]{Chr00},
\begin{align*}
\depth(M)&=\inf\{i|\Ext_R^i(R/\m,M)\not=0\}\\
&=\inf\{i|\Ext_R^i(R/\m,M')\not=0\}-n\\
&=\depth(M')-n\\
&=\Gid(M')-n\\
&=\Gid(M).
\end{align*}
See the corollaries for specific cases where Theorem \ref{connect} applies.
\end{rem}
\begin{proof}[Proof of Theorem \ref{connect}]
Recall that by definition of $\depth(\a,M)$ and $\cd(\a,M)$, we have that $c=\depth(\a,M)=\cd(\a,M)$ if and only if $H_\a^i(M)=0$ for all $i\not=c$.  Also $c\leq t$, since if $c>t$ we would violate Corollary \ref{loc_coh_of_gor_inj_is_zero_in_high_degs}. 

Let $M\to I\to U$ be any minimal complete injective resolution (see Construction \ref{construction_cir} for an explicit construction).   Apply $\Gamma_\a(-)$ to the map of complexes $I\to U$ to obtain the map of complexes $\Gamma_\a(I)\to \Gamma_\a(U)$ (recall that $\Gamma_\a(U)$ remains exact by Lemma \ref{Gamma_preserves_acy}).

Fix $\ell<c=\depth(\a,M)$.  We claim that $\Gamma_\a(I^\ell)=0$.  It will be enough to show that $\mu_R^\ell(\p,M)=0$ for all $\p\supseteq \a$ (if $\p\not\supseteq \a$, then  $\Gamma_\a(E(R/\p))=0$). So let $\p$ be any prime containing $\a$.  By \cite[Proposition 2.10]{FI01}, 
$\depth_R(\a,M)=\inf\{\depth_{R_\q}M_\q|\q\supseteq \a\},$ so 
$$\ell<\depth(\a,M)\leq \depth_{R_\p}M_\p=\inf\{i|\Ext_{R_\p}^i(\kappa(\p), M_\p)\not=0\}.$$
Therefore 
$\mu_R^\ell(\p,M)=\dim_{\kappa(\p)}\Ext_{R_\p}^\ell(\kappa(\p), M_\p)=0,$
and so $\Gamma_\a(I^\ell)=0$.

By minimality of $I$, $t+1$ is the minimal integer such that $\ker(I^{t+1}\to I^{t+2})$ is reduced Gorenstein injective. To see this, note that \cite[Proposition 2.3]{EJ95res} gives that $Z^t(I)$ is Gorenstein injective, and therefore $Z^{t+1}(I)$ is reduced by \cite[Theorem 10.1.4]{EJ00} and the proof of \cite[Proposition 10.1.8]{EJ00}.  Thus for $i\geq t+1$, $I^i\cong U^i$; in particular $Z^{t+1}(I)\cong Z^{t+1}(U)$, and henceforth we identify these modules, setting $N:=Z^{t+1}(I)\cong Z^{t+1}(U)$.  Note that as $N$ is reduced and Gorenstein injective, so is $\Gamma_\a(N)$.  We therefore have the following diagram (using that $\Gamma_\a(I^i)=0$ for $i<c$ as shown above):

\[\xymatrix{
\cdots\ar[r] & \Gamma_\a(U^{c-1})\ar[r] & \Gamma_\a(U^c)\ar[r] & \Gamma_\a(U^{c+1})\ar[r] & \cdots \ar[r] & \Gamma_\a(U^{t}) \ar[r] & \Gamma_\a(U^{t+1})\ar[r] & \cdots\\
\cdots \ar[r] & 0\ar[u] \ar[r] & \Gamma_\a(I^c)\ar[r]\ar[u] & \Gamma_\a(I^{c+1})\ar[r]\ar[u] & \cdots \ar[r] & \Gamma_\a(I^{t}) \ar[r]\ar[u] & \Gamma_\a(I^{t+1})\ar[r]\ar[u]^{\cong} & \cdots
}\]

Since $\Gamma_\a(N)$ is reduced Gorenstein injective, $\Gamma_\a(U^t)\to \Gamma_\a(N)$ is an injective cover.  Also, \cite[Theorem 10.1.4]{EJ00} gives $\Gamma_\a(I^t)\to \Gamma_\a(N)$ is an injective precover.  Therefore by definition of injective precovers, there exist maps $\Gamma_\a(U^t)\to \Gamma_\a(I^t)$ and $\Gamma_\a(I^t)\to \Gamma_\a(U^t)$ and a commutative diagram
\[\xymatrix{\
\Gamma_\a(U^t)\ar[r]\ar[dr] & \Gamma_\a(I^t)\ar[r]\ar[d] &\Gamma_\a(U^t)\ar[dl]\\
&\Gamma_\a(N)&
}\]
where since $\Gamma_\a(U^t)\to \Gamma_\a(N)$ is an injective cover and the diagram commutes, we must have the composition $\Gamma_\a(U^t)\to \Gamma_\a(I^t)\to \Gamma_\a(U^t)$ is an isomorphism, hence the first horizontal map is an injection and the second is a surjection, such that the composition is isomorphic to the identity on $\Gamma_\a(U^t)$.  We have therefore shown that $\Gamma_\a(U^t)$ appears as a direct summand of $\Gamma_\a(I^t)$.

Note that $\cosyz^cM\to \Sigma^c(I^{\geq c})\to\Sigma^c U$ is a minimal complete injective resolution of $\cosyz^cM$. Then by definition, $\Gstab_\a(\cosyz^cM)=Z^0\Gamma_\a(\Sigma^cU)=Z^c\Gamma_\a(U)$, so we have the following diagram, with exact rows:

\[\xymatrix{
0\ar[r] & \Gstab_\a(\cosyz^{c}M) \ar[r] & \Gamma_\a(U^c)\ar[r] & \Gamma_\a(U^{c+1})\ar[r] & \cdots \ar[r] & \Gamma_\a(U^{t}) \ar[r] & \Gamma_\a(N)\ar[r] &  0 \\
0\ar[r] & H_\a^c(M)\ar[r]\ar[u] & \Gamma_\a(I^c)\ar[r]\ar[u] & \Gamma_\a(I^{c+1})\ar[r]\ar[u] & \cdots \ar[r] & \Gamma_\a(I^{t}) \ar[r]\ar[u] & \Gamma_\a(N)\ar[r]\ar[u]^{=}\ar[r] & 0
}\]
Totalization induces an exact sequence:

\begin{align}\label{exact_seq_of_slc_and_clc}
0\to H_\a^c(M)\xrightarrow{\del^{c-1}} \begin{matrix} \Gstab_\a(\cosyz^{c}M) \\ \oplus \\\Gamma_\a(I^c)\end{matrix}\xrightarrow{\del^c} \begin{matrix}\Gamma_\a(U^c) \\ \oplus \\ \Gamma_\a(I^{c+1})\end{matrix} \xrightarrow{\del^{c+1}} \cdots \xrightarrow{\del^{t}} \begin{matrix}\Gamma_\a(U^{t}) \\ \oplus \\ \Gamma_\a(N)\end{matrix} \xrightarrow{\del^{t+1}} \Gamma_\a(N)\to 0,
\end{align}
where $\del^i$ is defined in the obvious way \cite[Proposition 1.4.14]{EJ00}.  Note that the complex 
$$0\to \cdots \to 0 \to \Gamma_\a(N)\xrightarrow{\pm \id} \Gamma_\a(N)\to 0$$
appears as a subcomplex of the exact sequence (\ref{exact_seq_of_slc_and_clc}), so we can quotient out by it to obtain another exact sequence.  We consider the cases of $c=t$ and $0\leq c\leq t-1$ separately. 

First suppose $c=t$.  After quotienting the exact sequence (\ref{exact_seq_of_slc_and_clc}) out by $0\to \Gamma_\a(N)\xrightarrow{\pm \id} \Gamma_\a(N)\to 0$, we obtain a short exact sequence
$$0\to H_\a^t(M)\to \begin{matrix}\Gstab_\a(\cosyz^tM) \\ \oplus \\ \Gamma_\a(I^t) \end{matrix} \to \Gamma_\a(U^t) \to 0.$$
Since $\Gamma_\a(-)$ preserves essential injections, $\Gamma_\a(I^t)\cong E_R(H_\a^t(M))$ and $\Gamma_\a(U^t)\cong E_R(\Gstab_\a(\cosyz^t M))$, and further since $\Gamma_\a(I^t)\to \Gamma_\a(U^t)$ is a split surjection, we obtain the desired split short exact sequence when $c=t$:
$$0\to H_\a^t(M)\to \Gstab_\a(\cosyz^tM) \oplus E_R(H_\a^t(M)) \to E_R(\Gstab_\a(\cosyz^t M)) \to 0.$$

Next suppose that $0\leq c\leq t-1$.  Quotienting out the exact sequence (\ref{exact_seq_of_slc_and_clc}) by $0\to \Gamma_\a(N)\xrightarrow{\pm \id} \Gamma_\a(N)\to 0$, we obtain the following exact sequence (we abuse notation and use the same names for the maps):
\begin{align}\label{gorinjres_of_H_c_a}
0\to H_\a^c(M)\xrightarrow{\del^{c-1}} \begin{matrix} \Gstab_\a(\cosyz^{c}M) \\ \oplus \\\Gamma_\a(I^c)\end{matrix}\xrightarrow{\del^c} \begin{matrix}\Gamma_\a(U^c) \\ \oplus \\ \Gamma_\a(I^{c+1})\end{matrix} \xrightarrow{\del^{c+1}} \cdots \xrightarrow{\del^{t-1}} \begin{matrix}\Gamma_\a(U^{t-1}) \\ \oplus \\ \Gamma_\a(I^{t})\end{matrix}\xrightarrow{\del^{t}} \begin{matrix}\Gamma_\a(U^{t})\end{matrix}\to 0.
\end{align}
Set $K:=\coker(\del^{c-1})$. If $\id_R M<\infty$, then $\Gid_RM=\id_RM\leq d$, hence $U=0$, so 
$$0\to K\to \Gamma_\a(I^{c+1})\to \cdots \to \Gamma_\a(I^t)\to 0$$
is a minimal injective resolution (as $\Gamma_\a(-)$ preserves minimal injective resolutions), so $\id_RK=t-c-1$ as desired.  Henceforth we assume that $\id_R M=\infty$ (equivalently, $\pd_R M=\infty$).

Since the injective module $\Gamma_\a(U^t)$ is a summand of the injective module $\Gamma_\a(I^t)$, there is an injective module $J$ such that $\Gamma_\a(I^t)\cong \Gamma_\a(U^t)\oplus J$.  Set $\pi:\Gamma_\a(I^t)\to J$ as the canonical surjection.  This allows us to cancel off the appearance of $0\to \Gamma_\a(U^t)\xrightarrow{\cong} \Gamma_\a(U^t)\to 0$ in the exact sequence (\ref{gorinjres_of_H_c_a}) to obtain an injective resolution for $K$:
$$0\to K \xrightarrow{\del^{c}} \begin{matrix}\Gamma_\a(U^c) \\ \oplus \\ \Gamma_\a(I^{c+1})\end{matrix} \xrightarrow{\del^{c+1}} \cdots \xrightarrow{\del^{t-2}} \begin{matrix}\Gamma_\a(U^{t-2}) \\ \oplus \\ \Gamma_\a(I^{t-1})\end{matrix}\xrightarrow{\del^{t-1}} \begin{matrix}\Gamma_\a(U^{t-1}) \\ \oplus \\ J \end{matrix}\to 0,$$
hence $\id_RK\leq t-1-c$.  To show $\id_RK=t-1-c$, it is enough to show that $\Ext_R^{t-c-1}(R/\m,K)\not=0$.  Apply $\Hom_R(R/\m,-)$ to the injective resolution of $K$ to obtain:

\[\xymatrix{
\cdots \ar[r] & {\begin{matrix}\Hom_R(R/\m,\Gamma_\a(U^{t-2})) \\ \oplus \\ \Hom_R(R/\m,\Gamma_\a(I^{t-1}))\end{matrix}}\ar[r]^{(\del^{t-1})_*} & {\begin{matrix}\Hom_R(R/\m,\Gamma_\a(U^{t-1})) \\ \oplus \\ \Hom_R(R/\m,J)\end{matrix}}\ar[r] & 0\ar[r]& \cdots 
}\]
where if $\Gamma_\a(\del_U)$ and $\Gamma_\a(\del_I)$ are the differentials on $\Gamma_\a(U)$ and $\Gamma_\a(I)$, respectively, and if $\Gamma_\a(\rho):\Gamma_\a(I)\to \Gamma_\a(U)$ is the map induced by the minimal complete injective resolution, then
$$(\del^{t-1})_*=\left(\begin{matrix}(\Gamma_\a(\del_U^{t-2}))_* & (\Gamma_\a(\rho^{t-1}))_* \\ 0 & (\pi\circ\Gamma_\a(\del_I^{t-1}))_* \end{matrix}\right).$$
Since $\depth(M)=\Gid_RM=t>t-1$, Lemma \ref{mu_vanishes_below_grade} gives that $E(R/\m)$ does not appear in $I^{t-1}$, and hence also not in $\Gamma_\a(I^{t-1})$.  Therefore $\Hom_R(R/\m,\Gamma_\a(I^{t-1}))=0$, so $(\Gamma_\a(\rho^{t-1}))_*=0$.  Also, as $\Gamma_\a(I)$ and $\Gamma_\a(U)$ are both minimal complexes, $\Hom_R(R/\m,-)$ applied to either of their differentials becomes the zero map (see Remark \ref{min_inj_same_as_zero_maps}), hence $(\del^{t-1})_*=0$.

In order to show that $\Ext_R^{t-c-1}(R/\m,K)\not=0$, it is therefore enough to find a nonzero element in $\Hom_R(R/\m,\Gamma_\a(U^{t-1}))$.  Since we are in the case where $\pd_RM=\infty$, we have $E(R/\m)$ appears as a summand of $U^{t-1}$ \cite[Theorem 10.3]{AM02}, and therefore (since $\a\subseteq \m$) appears as a summand of $\Gamma_\a(U^{t-1})$, hence $\Hom_R(R/\m,\Gamma_\a(U^{t-1}))\not=0$ \cite[Theorem A.20]{24hours}.  Therefore $\Ext_R^{t-c-1}(R/\m,K)\not=0$, and so $\id_R K=t-1-c$.  Noting that $\Gamma_\a(I^c)\cong E_R(H_\a^c(M))$, we have the desired short exact sequence when $c\leq t-1$:

$$0\to H_\a^c(M)\to \Gstab_\a(\cosyz^cM)  \oplus E_R(H_\a^c(M)) \to K \to 0.$$
\end{proof}

We highlight a special case of the previous theorem:
\begin{cor}\label{cor_to_connect}
Let $(R,\m)$ be a Gorenstein local ring of Krull dimension $d$.  Suppose $M\not=0$ is a MCM $R$-module, such that $c=\depth(\a,M)=\cd(\a,M)$.  Then there exists a short exact sequence
$$0\to H_\a^c(M)\to \Gstab_\a(\cosyz^{c}M)\oplus E_R(H_\a^c(M))\to K\to 0,$$
where $\id_RK<\infty$.  Moreover, when $0\leq c\leq t-1$, we have $\id_RK=t-c-1$ and when $c=t$, the sequence splits and $K\cong E_R(\Gstab_\a(\cosyz^tM))$.
\end{cor}

\begin{example}
Let $(R,\m)$ be a local Gorenstein ring of finite Krull dimension and $M\not=0$ a MCM $R$-module.  If $\a$ is any ideal generated up to radical by a regular sequence, Theorem \ref{connect} applies.
\end{example}

\begin{rem}\label{finite_inj_dim_trivial}
Let $R$ be Gorenstein of Krull dimension $d$ and $N$ be any $R$-module.  Then $N\simeq 0$ in $\sGInj(R)$ if and only if $\id_R N<\infty$.
\end{rem}

\begin{cor}\label{slc_and_classical_coincide}
Let $R$ be a Gorenstein local ring of dimension $d$, $M\not=0$ a MCM $R$-module where $\Gid_RM=\depth M$ and $\a\subset R$ an ideal satisfying $c=\depth(\a,M)=\cd(\a,M)$. Then we have an isomorphism in $\sGInj(R)$,
$$H_\a^c(M)\simeq \Gstab_\a(\cosyz^cM).$$ 
\end{cor}
\begin{proof}
Apply Remark \ref{finite_inj_dim_trivial} to Theorem \ref{connect}.
\end{proof}

This also recovers a result of Zargar and Zakeri in the case of a Gorenstein ring:
\begin{cor}\cite{ZZ13}
Let $R$ be a Gorenstein local ring of dimension $d$, and $M$, $\a$, and $c$ be as in Theorem \ref{connect}.  Then
$$\Gid_R H_\a^c(M)=\Gid_RM-c.$$
\end{cor}
\begin{proof}
This follows immediately from Theorem \ref{connect}.
\end{proof}

Recall that a MCM approximation of a finitely generated module $N$ is a short exact sequence $0\to I\to M\to N\to 0$, where $\id_RI<\infty$ and $M$ is MCM.  Often we just refer to $M$ as the MCM approximation of $N$.

Dually, for an artinian module $N$, a short exact sequence of the form $0\to N\to G\to P\to 0$, where $G$ is Gorenstein injective and $\pd_RP<\infty$ is called a Gorenstein injective approximation of $N$ \cite[section 7]{Kra05}.  Therefore, in light of Theorem \ref{connect}, we have:

\begin{cor}\label{ginj-approx}
The short exact sequence given in Theorem \ref{connect} is a Gorenstein injective approximation of $H_\a^{c}(M)$.
\end{cor}


\newcommand{\etalchar}[1]{$^{#1}$}

\end{document}